\documentclass[10pt]{article}
\usepackage{amsthm, amsmath, amssymb, amsfonts, url, booktabs, tikz, setspace, fancyhdr, bm}
\usepackage[hidelinks]{hyperref}
\usepackage{hyperref, enumerate}
\usepackage[shortlabels]{enumitem}
\usepackage[babel]{microtype}
\usepackage[english]{babel}
\usepackage[capitalise]{cleveref}
\usepackage{comment}
\usepackage{scalerel}

\usepackage{soul,xcolor}
\usepackage{nccmath}
\usepackage{color}
\usepackage{mathtools}
\usepackage{authblk}

\DeclarePairedDelimiter{\floor}{\lfloor}{\rfloor}
\usepackage{amsthm}
\usepackage{mathrsfs} 

\newtheorem*{remark*}{Remark}
\newtheorem*{notation*}{Notation}
\newtheorem{theorem}{Theorem}[section]

\newtheorem{proposition}[theorem]{Proposition}

\newtheorem{lemma}[theorem]{Lemma}
\newtheorem{claim}[theorem]{Claim}
\newtheorem{observation}[theorem]{Observation}

\newtheorem{conjecture}[theorem]{Conjecture}

\numberwithin{equation}{section}

\Crefname{claim}{Claim}{Claims}

\def \G {\Gamma}

\def \F {\mathcal{F}}
\def \R {\mathbb{R}}
\def \a {\alpha}
\def \b {\beta}
\def \e {\varepsilon}
\def \d {\delta}
\def \D {\Delta}
\def \a {\alpha}

\def \t {\tau}
\def \T {\theta}
\def \g {\gamma}
\def \l {\lambda}
\def \h {\eta}

\def \E {\mathbb{E}}
\def \P {\mathbb{P}}

\DeclareMathOperator{\dist}{dist}

\allowdisplaybreaks

\newcounter{propcounter}

\title{Large rainbow matchings in edge-colored graphs}
\author[1]{Debsoumya Chakraborti\thanks{Supported by the Royal Society and the Institute for Basic Science (IBS-R029-C1).}} 
\author[2]{Po-Shen Loh\thanks{Supported in part by National Science Foundation CAREER Grant DMS-1455125.}}
\affil[1]{\small Department of Mathematics, University College London (UCL), London,~UK}
\affil[2]{\small Department of Mathematical Sciences, Carnegie Mellon University, Pittsburgh,~USA}
\affil[ ]{\small Email:
\texttt{d.chakraborti@ucl.ac.uk},
\texttt{ploh@cmu.edu} }

\begin{document}
\setstcolor{red}
\maketitle
\begin{abstract}
A subgraph of an edge-colored graph is called \emph{rainbow} if all of its edges have distinct colors. There has been much research on the topic of finding a large rainbow matching in a properly edge-colored graph, where a proper edge-coloring is a coloring of the edge set such that no same-colored edges are incident. Gao, Ramadurai, Wanless, and Wormald proved that in every proper edge-coloring of a graph with $n$ colors where each color appears at least $n+o(n)$ times, there is always a rainbow matching using every color. We strengthen this result by simultaneously relaxing three conditions: (i) we lift the condition on the number of colors and allow any finite number of colors and instead, put a weaker condition requiring the maximum degree of the graph to be at most $n$, (ii) we relax the proper coloring condition and require that the graph induced by each of the colors have maximum degree $o(n)$, and (iii) we work in a more general setting of multigraphs allowing edge multiplicities to be $o(n)$. 

As an application of this result, we show that for every proper edge-coloring of a graph with $2n+o(n)$ colors where each color appears at least $n$ times, there is always a rainbow matching of size $n$. Aharoni and Berger conjectured that $2n+o(n)$ can be replaced by $n+1$ in this statement. We dispute this conjecture with an explicit construction.
\end{abstract}

\section{Introduction}

\subsection{State of the art}
Transversals in Latin squares have been a central topic of study in combinatorics, dating back to the work of Euler \cite{Euler} in the 18th century, who studied conditions under which Latin squares can be decomposed into transversals. For a survey of transversals in Latin squares, see, e.g., \cite{W}. One of the central and long-standing conjectures in this field is the following, which is attributed to Brualdi, Ryser, and Stein.

\begin{conjecture} [Brualdi and Ryser \cite{BR}, Ryser \cite{Ryser}, and Stein \cite{S}] \label{brs}
Every $n \times n$ Latin square has a partial transversal of size $n-1$ if $n$ is even and of size $n$ if $n$ is odd.
\end{conjecture} 

This conjecture translates into the graph-theoretic statement: ``In every proper edge-coloring of the complete bipartite graph $K_{n,n}$ with $n$ colors, there is a rainbow matching of size $n-1$ if $n$ is even and of size $n$ if $n$ is odd''. To see this connection, refer to \cite{P} or \cite{W}. Brouwer, De Vries, and Wieringa \cite{BVW} and independently Woolbright \cite{Woolbright} proved an asymptotic version of \Cref{brs}. The error term was later improved by Hatami and Shor \cite{HS}, and subsequently by Keevash, Pokrovskiy, Sudakov, and Yepremyan \cite{KPSY}. Recently, Montgomery \cite{M} settled this conjecture for all sufficiently large even $n$.

In this paper, we study various generalizations of \Cref{brs} in the setting of rainbow matchings in edge-colored graphs. We advise the interested readers to see \cite{ABriggs} for a recent survey on various extensions of this conjecture. Aharoni and Berger \cite{AB} conjectured the following generalization of \Cref{brs}. Throughout this paper, multigraphs permit parallel edges but not loops. A \emph{proper edge-coloring} of a multigraph refers to a coloring of the edge set such that every pair of incident edges receives distinct colors.

\begin{conjecture}[Aharoni and Berger \cite{AB}] \label{ab}
Let $G$ be a properly edge-colored bipartite multigraph with $n$ colors, having at least $n+1$ edges of each color. Then $G$ has a rainbow matching using every color.
\end{conjecture}

An easy greedy argument shows that having $2n-1$ edges in each color suffices in the above conjecture. After subsequent efforts by several authors (see, e.g., \cite{ACH,AKZ',CE,HJ,KZ,P'}), an asymptotic version of \Cref{ab} is established in \cite{CPS,P}, where the conclusion of the conjecture is shown to hold if there are at least $n+o(n)$ edges of each color. It is very natural to generalize \Cref{ab} for non-bipartite graphs. Gao, Ramadurai, Wanless, and Wormald \cite{GRWW} mentioned the following conjecture, which was also suggested by Aharoni, Berger, Chudnovsky, Howard, and Seymour \cite{ABCHS}. 

\begin{conjecture} [Gao, Ramadurai, Wanless, and Wormald \cite{GRWW}] \label{abc}
Let $G$ be a properly edge-colored multigraph with $n$ colors, having at least $n+2$ edges of each color. Then $G$ has a rainbow matching using every color.
\end{conjecture}

We remark here that one cannot replace $n+2$ by $n+1$ in \Cref{abc} (to see this, refer to \cite{CPS,GRWW}). After a couple of notable works progressing towards \Cref{abc} by Gao, Ramadurai, Wanless, and Wormald \cite{GRWW}, and by Keevash and Yepremyan \cite{KY}, recently Correia, Pokrovskiy, and Sudakov \cite{CPS} established an asymptotic version of \Cref{abc}. Gao, Ramadurai, Wanless, and Wormald proved the following asymptotic version of \Cref{abc} for simple graphs.

\begin{theorem}[Gao, Ramadurai, Wanless, and Wormald \cite{GRWW}] \label{gao}
There exist $n_0, \xi > 0$ such that whenever $n\ge n_0$, for every graph $G$ that is properly edge-colored with $n$ colors such that there are at least $(1+n^{-\xi})n$ edges of each color, there is a rainbow matching of $G$ using every color. 
\end{theorem}

We remark here that in \cite{GRWW}, the above theorem is shown to be true if the number of colors is at most $n^{1+\d}$ with $0 \le \d < 1/4$. One of the main contributions of this paper is to replace this condition on the number of colors with a weaker maximum-degree condition in the above result (see \Cref{th1} in the next subsection). We further remark that in \cite{GRWW}, a version of \Cref{gao} was proved for multigraphs with bounded edge multiplicities. Our \Cref{th1} assumes a more general edge multiplicity condition compared to that, and also allows non-proper edge-colorings. Yet another motivation for such general \Cref{th1} comes from the following conjecture by Aharoni and Berger \cite{AB}, which allows the edge-coloring to be non-proper.

\begin{conjecture} [Aharoni and Berger \cite{AB}] \label{a}
Let G be an edge-colored bipartite multigraph with maximum degree $n$. If every color appears on at least $n + 1$ edges, then $G$ has a rainbow matching using every color.
\end{conjecture}

Gao, Ramadurai, Wanless, and Wormald refuted \Cref{a} in \cite{GRWW} by constructing an example where there are color classes (color class refers to the subgraph formed by the edges of that color) with maximum degree linear in $n$. In this context, it is interesting to ask if such a statement holds for edge-coloring with the additional assumption of bounded maximum degree in each color class. We answer this positively with our generalization (\Cref{th1}) of \Cref{gao} that implies an asymptotic version of \Cref{a} for edge-coloring of not-necessarily-bipartite multigraphs with only a maximum degree $o(n)$ assumption on each color class and an edge multiplicity $o(n)$ assumption on $G$. This is also a generalization of \Cref{ab} because for any proper edge-coloring of a multigraph, the maximum degrees of all color classes are one. 

We now move on to a related but slightly different problem. All the problems discussed so far focused on the minimum number of edges in each color class to ensure a rainbow matching using all the colors. Alternatively, we can insist that each color class has exactly $n$ edges and ask how many colors are needed to ensure a rainbow matching of size $n$. In this context, the first result appeared by Drisko \cite{D} in 1998, which some authors later revisited (see, e.g., \cite{AB,ABKZ,AKZ}). On this topic, Bar\'at, Gy\'arf\'as, and S\'ark\"ozy \cite{BGS} suggested the following conjecture.

\begin{conjecture} [Bar\'at, Gy\'arf\'as, and S\'ark\"ozy \cite{BGS}] \label{bgs}
Let $G$ be a properly edge-colored multigraph with $2n-t_n$ colors and exactly $n$ edges of each color, where $t_n = 0$ for even $n$ and $t_n = 1$ for odd $n$. Then $G$ has a rainbow matching using $n$ colors.
\end{conjecture}

For the best-known results on this conjecture, the readers are advised to refer to \cite{ABCHS,ABKK}. Aharoni and Berger considered \Cref{bgs} in a simple graph, and made the following general conjecture (Conjecture~3.2 in \cite{AB}) for hypergraphs.

\begin{conjecture} [Aharoni and Berger \cite{AB}] \label{aharoni}
Let $s\le t$ and $G$ be a properly edge-colored $r$-uniform hypergraph with $2^{r-2}(s-1) + 2$ colors and $t$ edges of each color. Then, $G$ has a matching with $t$ edges on which at least $s$ colors appear.
\end{conjecture}

Aharoni and Berger \cite{AB} mentioned that the $r=2$ and $s=t$ case of \Cref{aharoni} gives us a generalization of \Cref{brs}. However, we disprove this conjecture in this case (see \Cref{th2} below) and make progress toward an upper bound on the required number of colors, which we discuss in the next subsection. To be more precise, this upper bound is obtained by establishing an asymptotic version of \Cref{bgs} when we restrict ourselves to simple graphs $G$ (see \Cref{th} below).

\subsection{Main results}
We start with our first result, which strengthens \Cref{gao}.

\begin{theorem} \label{th1}
There exist $\e_0, \xi > 0$ such that the following hold whenever $0 < \e \le \e_0$. Suppose $G$ is an edge-colored multigraph with maximum degree at most $n$ such that there are at least $(1 + \e^{\xi})n$ edges of each color. Furthermore, suppose that the multiplicity of $G$ is at most $\e n$ and at most $\e n$ edges of the same color are incident to any vertex. Then there is a rainbow matching in $G$ that uses every color.
\end{theorem}

One can recover \Cref{gao} from above by taking $\e = 1/n$. Moreover, in contrast to \Cref{gao}, we do not need any bound on the number of colors; instead, we require a weaker condition that the maximum degree is at most $n$. (Note that in \Cref{gao}, a proper edge-coloring of $G$ with $n$ colors ensures the maximum degree to be at most $n$.) Furthermore, \Cref{th1} does not require the edge-coloring to be proper, but the subgraph formed by each color still must have maximum degree $o(n)$. To prove \Cref{th1}, we utilize the `nibble' technique (introduced in R\"odl's pioneering work \cite{R} and developed by various authors) combined with the local lemma. For a single `nibble' step, we use a natural randomized algorithm similar to the one used by Gao, Ramadurai, Wanless, and Wormald \cite{GRWW}. 

We next state our result that disproves \Cref{aharoni} for $r=2$ and $s=t$. 

\begin{proposition} \label{th2}
For all even $n\ge 2$, there exists a graph that is properly edge-colored with $n+1$ colors such that there are $n$ edges of each color, but no rainbow matching of size $n$. 
\end{proposition}

Thus, it is natural to find the maximum number $f(n)$ that can be written instead of $n+1$ in the above proposition. \Cref{th2} proves that $f(n) \ge n+1$ for all even $n\ge 2$. We make progress in the upper bound and establish that $f(n) \le 2n + o(n)$. This provides an asymptotic version of \Cref{bgs} in the case of simple graphs. However, we believe that our lower bound on $f(n)$ is closer to the truth. 

\begin{theorem} \label{th}
There exist $n_0, \xi >0$ such that whenever $n \ge n_0$, for every graph $G$ that is properly edge-colored with at least $2(1+n^{-\xi})n$ colors such that there are at least $n$ edges of each color, there is a rainbow matching of $G$ using $n$ colors.
\end{theorem}

We remark that the proof of this result is quite short and uses \Cref{th1} as a black box. We now show that \Cref{th} is the best possible with respect to the number of edges in each color. We require each color to appear at least $n$ times, since if every color appears only $n-1$ times, then it is possible to construct examples with no matching of size $n$ (and hence no rainbow matching of size $n$), regardless of the number $m$ of colors. To see this, consider the graph that is a disjoint union of $n-1$ copies of the complete bipartite graph $K_{1,m}$ and color each $K_{1,m}$ using all the colors exactly once. Although each color appears as a matching with $n-1$ edges, the graph does not contain any matching with $n$ edges. 

Since this paper was first arXived, a hypergraph generalization of \Cref{th1} and a reproof of \Cref{th} appeared in \cite{DP} by Delcourt and Postle. However, their proofs are more involved as they prove more general results. Moreover, they prove that it suffices to have $n+o(n)$ colors when the underlying graph is bipartite in \Cref{th}, thereby addressing a question we suggested in a previous version of this paper. 

\medskip
\noindent\textbf{Organization.} This paper is organized as follows. The next section contains some standard probabilistic tools, which we use throughout this paper. In \Cref{sec:3}, we disprove \Cref{aharoni} for $r=2$ and $s=t$ by proving \Cref{th2}, and we also prove \Cref{th} assuming the statement of \Cref{th1}. We finally prove \Cref{th1} in \Cref{sec:4}.

\medskip
\noindent\textbf{Notation.} For natural numbers $n$, we denote by $[n]$ the set $\{1,2,\ldots,n\}$. Throughout this paper, for brevity, we omit the floor and ceiling signs, as they do not affect the underlying analysis. For convenience, we say an event $A_n$ happens \emph{with very high probability} (in short, w.v.h.p.) to mean that the probability that $A_n$ holds with probability at least $1-n^{-\omega(1)}$, where $\omega(1)$ is any function whose ratio with $1$ tends to infinity as $n$ tends to infinity. This notion is useful because if there is a collection of $n^{O(1)}$ events where each of them occurs w.v.h.p, then by the union bound, all of the events simultaneously occur w.v.h.p.

We use standard graph theoretic notations. Multigraphs are allowed to contain parallel edges between the same pair of vertices, but they do not contain loops. Two edges are called \emph{parallel} if they have the same pair of endpoints. Consider a multigraph $G$. We denote its vertex set by $V(G)$ and its edge (multi-)set by $E(G)$. We denote by $e(G)$ the number of edges in $G$, i.e., $e(G) = |E(G)|$. For vertices $u,v\in V(G)$, the \emph{multiplicity} of the pair $(u,v)$ is the number of parallel edges between $u$ and $v$. The \emph{multiplicity} of $G$ is defined as the maximum multiplicity of any pair of vertices. For $V\subseteq V(G)$, we denote by $G[V]$ the subgraph of $G$ induced by $V$. For $v\in V(G)$, we denote by $G\setminus v$ the graph $G[V(G)\setminus \{v\}]$. For $v\in V(G)$, the neighborhood of $v$ is denoted by $N_G(v)$ and the degree of $v$ is denoted by $d_G(v)$.

We now consider an edge-colored multigraph $G$. We denote by $\mathscr{C}(G)$ the set of all colors used to color the edges of $G$. For a set of colors $\mathscr{C}$, the multigraph $G$ is called $\mathscr{C}$-rainbow if $\mathscr{C}(G) = \mathscr{C}$ and the edges of $G$ are colored with pairwise distinct colors from $\mathscr{C}$. For $v\in V(G)$ and $c\in \mathscr{C}(G)$, we say $v$ \emph{meets} $c$ if there is an edge incident to $v$ that is colored with $c$. Whenever $G'$ is a subgraph of $G$, we equip $G'$ with the edge-coloring obtained by restricting the edge-coloring of $G$ to $G'$.

\section{Preliminaries}
To prove our results, we will use the following standard probabilistic tools throughout this paper. We start with a couple of concentration inequalities. 

\begin{theorem} [Chernoff bound, see \cite{H,MU}] \label{chernoff}
Let $X = \sum_{i=1}^n X_i$, where $X_i = 1$ with probability $p_i$ and $X_i = 0$ with probability $1 - p_i$, and all $X_i$ are independent. Let $\mu = \E(X) = \sum_{i=1}^n p_i$. Then for any $0 \le \l \le \mu$, 
\[
\P\left[|X - \mu| \ge \l\right] \le 2 \exp\left(-\frac{\l^2}{3\mu}\right).
\]
Moreover, for $\l \ge 8\mu$, we have $\P[X\ge \l] \le \exp(-\l)$.
\end{theorem}

To keep track of the evolution of certain parameters in our randomized algorithm for the proof of \Cref{th1}, we will often use Talagrand's inequality. We need a couple of definitions to state it. For a random variable $X : \Omega \rightarrow \R$ with $\Omega = \prod_{i=1}^n \Omega_i$ being a product probability space, we say that
\begin{itemize}
\item $X$ is \emph{$\ell$-Lipschitz}, if for every $\omega \in \Omega$, changing $\omega$ in any single coordinate affects the value of $X(\omega)$ by an additive factor of at most $\ell$.
\item $X$ is \emph{$r$-certifiable}, if for every $\omega \in \Omega$ and $s \in \R$ such that $X(\omega) \ge s$, the following holds. There exists a set $I \subseteq \{1,\ldots,n\}$ of size at most $rs$ such that every $\omega' \in \Omega$ that agrees with $\omega$ on the coordinates indexed by $I$ also satisfies that $X(\omega') \ge s$.
\end{itemize}

\begin{theorem} [Talagrand's inequality, see \cite{MRbook,MR}] \label{talagrand}
Suppose that $X$ is an $\ell$-Lipschitz and $r$-certifiable non-negative random variable defined on a product probability space. Then, for every $\l\ge 0$, we have 
\[
\P\left[|X - \E(X)| > \l + 20\ell\sqrt{r\E(X)} + 64\ell^2r\right] \le 4 \exp\left(-\frac{\l^2}{8\ell^2r(\E(X)+\l)}\right).
\]
\end{theorem}

Since the number of vertices and colors in $G$ in \Cref{th1} may be arbitrarily large, simple union bounds are often insufficient to control certain global parameters in the analysis of our randomized algorithm. Accordingly, we use the following version of the standard local lemma in our proof. 

\begin{theorem} [Lov\'asz local lemma, see \cite{AS}] \label{lll}
Let $A_1, \ldots, A_n$ be events in a probability space. Suppose that there exist $p$ and $d$ such that all $\P[A_i] \le p$, and each event $A_j$ is mutually independent of all the other events $\{A_i\}$ except at most $d$ of them. If $ep(d+1) \le 1$, then with positive probability none of the events $A_1, \ldots, A_n$ occur.
\end{theorem}

We next give a weaker bound for \Cref{th1}, which will be used in its proof.

\begin{proposition} \label{useAlon}
Suppose $G$ is an edge-colored multigraph with maximum degree at most $n$, with at least $4en$ edges of each color. Then, there is a rainbow matching in $G$ that uses every color. 
\end{proposition}

To prove \Cref{useAlon}, we use the following result of Alon \cite{A88}, which can be easily proved using the local lemma (i.e., \Cref{lll}). For the best possible result, the readers are advised to see \cite{H01}.

\begin{proposition}[\cite{A88}] \label{Alon}
let $G$ be a multipartite graph with maximum degree at most $\D$, whose parts $V_1, \ldots, V_r$ all have size $|V_i| \ge 2e\D$. Then, $G$ has an independent set containing exactly one vertex from each $V_i$.
\end{proposition} 

\begin{proof}[{\bf Proof of \Cref{useAlon}}]
Let $r$ be the number of colors in $G$, and $E_i$ denote the set of all edges of the $i$-th color. Consider the $r$-partite graph $\Gamma$ where the $i$-th part is $E_i$ (i.e., the edges in $G$ are vertices of $\Gamma$), and two vertices in $\Gamma$ are adjacent if the corresponding edges in $G$ are incident to each other. Thus, the maximum degree of $\Gamma$ is at most $2n$, and each part contains at least $2e \cdot 2n$ vertices. Thus, by \Cref{Alon} with $\D = 2n$, the graph $\Gamma$ has an independent set containing exactly one vertex from each $E_i$. Observe that such an independent set in $\Gamma$ corresponds with a rainbow matching in $G$ using every color. This completes the proof of \Cref{useAlon}.
\end{proof}

\section{Lower and upper bounds for Aharoni-Berger conjecture} \label{sec:3}
We prove \Cref{th2,th} in this section. 

\subsection{Construction for \Cref{th2}}
\begin{proof}[{\bf Proof of \Cref{th2}}]
Fix an even $n\ge 2$. Consider a vertex set $A \cup B$, where $A$ and $B$ have $n$ vertices each, and recognize each of $A$ and $B$ by the group $\mathbb{Z}_n$. Consider a graph $G$ on the vertex set $A \cup B$ where for each $j \in \mathbb{Z}_n$, introduce a color $j$ with $n$ edges such that each $a \in A$ is adjacent to $a+j \in B$. Thus, we have $n$ colors, each of which is a matching of size $n$. First, we prove that there is no rainbow matching using all of these $n$ colors. For the sake of contradiction, assume that we have such a rainbow matching and fix such a matching. All colors and vertices of $G$ have to participate in such a matching. Let $a_jb_j$ denote the edge in color $j \in \mathbb{Z}_n$ in the rainbow matching, where $a_j \in A$ and $b_j \in B$. Clearly, we have 
\[
\sum_{j \in \mathbb{Z}_n} a_j = \sum_{j \in \mathbb{Z}_n} b_j = \sum_{j \in \mathbb{Z}_n} j = \frac{n(n-1)}{2}.
\]
But by definition, we should have $\sum_{j \in \mathbb{Z}_n} (b_j - a_j) = \sum_{j \in \mathbb{Z}_n} j$, which is a contradiction for even $n$. 

Now let us introduce an extra $(n+1)$-st color in $G$, whose edges are the union of an arbitrary perfect matching in $A$ and an arbitrary perfect matching in $B$. This new color clearly has $n$ edges. Note that all the matchings are still disjoint because any matching considered before contains edges between $A$ and $B$ only. We claim that even after introducing this extra color, we still cannot have a rainbow matching of size $n$. For the sake of contradiction, assume that we have such a rainbow matching of size $n$. Clearly, the extra color has to participate in that because we could not do it before. Without loss of generality, the edge with the new color in the rainbow matching is between two vertices $u,v \in A$. Now it is impossible to have a rainbow matching of size $n-1$ in $G[(A\cup B)\setminus \{u,v\}]$ using the rest of the colors because each such edge uses a distinct vertex from each of $A$ and $B$. Hence, we have a contradiction. This proves \Cref{th2}.
\end{proof}

\subsection{Proof of \Cref{th}}
We first establish a crude upper bound on the number of colors needed in \Cref{th} in order to prove the theorem in full generality. For this, we prove the following weaker version of \Cref{th} using \Cref{th1}.

\begin{proposition} \label{weakerbound}
There exists $n_0$ such that whenever $n\ge n_0$, for every graph $G$ that is properly edge-colored with at least $4n$ colors such that there are at least $n$ edges of each color, there is a rainbow matching of $G$ using $n$ colors.
\end{proposition}

\begin{proof}
To prove this proposition, we start by showing that it holds if $4n$ is replaced by $2n^2$. Suppose we have a properly edge-colored graph $G$ with $2n^2$ colors, where each color appears $n$ times. Consider a maximum size rainbow matching $M = \{e_1, e_2, \ldots, e_m\}$ of $G$. Assume for the sake of contradiction that $|M| = m \le n-1$. Let $C\subseteq \mathscr{C}(G)$ denote all the colors used to color the edges connecting two vertices in $V(M)$. Clearly, $|C| \le \binom{2m}{2} < 2n^2 - 2n$. Thus, there are $2n$ colors $c_1, c_2, \ldots, c_{2n}\in \mathscr{C}(G)$ which do not appear in any of the edges between two vertices in $V(M)$. If there is an edge with any of these $2n$ colors that does not use any of the vertices in $V(M)$, then we are already done because we can obtain a larger rainbow matching by adding that edge to $M$. Hence, all the edges using colors from $\{c_1, c_2, \ldots, c_{2n}\}$ uses exactly one vertex in $V(M)$. By a simple application of the pigeonhole principle, for every color $c\in \{c_1, c_2, \ldots, c_{2n}\}$, there is an edge $e \in M$ such that there are two edges with color $c$, which are incident to $e$. By another application of the pigeonhole principle, there exists an edge $e \in M$ such that there are three colors (without loss of generality, say $c_1$, $c_2$, and $c_3$) with the property that there are two edges with each color in $\{c_1, c_2, c_3\}$ such that all six of these edges are incident to $e$. Now, it is easy to find two edges among these six edges such that they are not incident to each other and have distinct colors. By observing that one obtains a larger rainbow matching by adding these two edges and subtracting $e$ from $M$, we establish that \Cref{weakerbound} holds with $4n$ replaced by $2n^2$. 

We next show that $4n$ colors suffice in \Cref{weakerbound} for all sufficiently large $n$, as promised. For this, let $n$ be sufficiently large for the remainder of this proof. Suppose we have a properly edge-colored graph $G$ with $4n$ colors, where each color appears $n$ times. Let $G_0 = G$. For $i\in [|V(G)|]$, we inductively define $d_i$ to be the maximum degree of $G_{i-1}$, then choose $v_i$ to be a vertex of degree $d_i$ in $G_{i-1}$, and finally define $G_i$ to be the graph $G_{i-1}\setminus v_i$. Observe that $d_1 \ge d_2 \ge \cdots \ge d_{|V(G)|}$. We define $k$ to be the minimum integer such that $k\in [|V(G)|]$ and $d_k \le 3(n - k)$, if it exists. We split the proof into two cases based on the value of $k$. 

\medskip
\noindent\textbf{Case 1:} Suppose $k > n-\sqrt{n}$ or $k$ does not exist. In the graph $G_{n-\sqrt{n}}$, there are at least $\sqrt{n}$ edges of each color, and the number of colors is $4n \ge 2\left(\sqrt{n}\right)^2$. Thus, using the version of \Cref{weakerbound} established at the beginning of this proof with $4n$ replaced by $2n^2$, we can find a rainbow matching $M$ using $\sqrt{n}$ colors in the graph $G_{n-\sqrt{n}}$. Now, we can greedily add an edge incident to each of the vertices in $\{v_i : i\in [n-\sqrt{n}]\}$ to $M$ to get a rainbow matching of size $n$. This can be done inductively as follows. Note that the number of colors used in $M$ is $\sqrt{n}$, and the number of vertices in $M$ is $2\sqrt{n}$. Let $M_{n-\sqrt{n}} = M$. For some $i\in [n - \sqrt{n}]$, we assume that there is a rainbow matching $M_i$ of size $n-i$ such that $v_j\notin V(M_i)$ for every $j\in [i]$. Then, one can pick a neighbor $u_i$ of $v_i$ in the graph $G_{i-1}$ out of more than $3 (n-i)$ neighbors so that neither the color of the edge $u_iv_i$ nor the vertex $u_i$ is used in the rainbow matching $M_i$ built yet, and then we add the edge $u_iv_i$ to $M_i$ to obtain a rainbow matching $M_{i-1}$ of size $n-i+1$ such that $v_j\notin V(M_i)$ for every $j\in [i-1]$ (the last part follows because $u_i\in V(G_{i-1})$ and thus $u_i\notin \{v_1,\ldots,v_{i-1}\}$). This inductive process builds a rainbow matching $M_0$ of size $n$ in $G$.

\medskip
\noindent\textbf{Case 2:} Suppose $k \le n-\sqrt{n}$. In the graph $G_k$, there are at least $n-k$ edges of each color, and the maximum degree of $G_k$ is at most $3(n-k)$. Now partition the set of $4n$ colors into $n$ subsets containing $4$ colors each and merge $4$ colors from each subset. This ensures that the total number of new colors after merging is $n$, each new color has at least $4(n-k)$ edges in $G_k$, and the subgraph of $G_k$ induced by each new color has maximum degree at most $4$. Since the maximum degree of $G_k$ is at most $3(n-k)$, we apply \Cref{th1} to obtain a rainbow matching in $G_k$ using the $n$ new colors, which yields a rainbow matching of size $n$ in $G$. 
\end{proof}

We are now ready to prove \Cref{th}.
\begin{proof}[{\bf Proof of \Cref{th}}]
Let $\xi >0$ be sufficiently small and $n$ be sufficiently large relative to $\xi$. Let $\e = n^{-\xi}$. Assume the graph $G$ is properly edge-colored with $2(1+\e)n$ colors such that there are $n$ edges of each color. (Note that this can be assumed instead of assuming that there are at least $n$ edges of each color, without loss of generality.) Define $\T = \e/2$. Let $A\subseteq V(G)$ be a maximal set satisfying $|A|\le (1-\T)n$ and every vertex in $A$ has degree at least $2(1+\T)n$. We will split into two cases based on the size of $A$. 

\medskip
\noindent\textbf{Case 1:} Suppose $|A| = (1-\T)n$. Let $G'$ denote the subgraph of $G$ after deleting all the edges incident to $A$. We have at least $\T n$ edges in each color in $G'$, and the number of colors is at least $2n$. Thus, by applying \Cref{weakerbound}, we find a rainbow matching $M$ of size $\T n$ in $G'$. Now consider the subgraph $G''$ of $G$ after discarding all the vertices in $V(M)$ and all the edges colored with a color used in $M$. For every fixed $v \in A$, we discard at most $3\T n$ edges incident to $v$, and so, $v$ still has at least $2(1+\T)n - 3\T n - (1-\T)n = n$ neighbors outside of $A$ in $G''$. Let $B = V(G)\setminus A$ and $C$ denote the set of all colors. We now consider an auxiliary bipartite graph $H$ with the vertex set bipartition $(B,C)$ that is edge-colored with set of colors $A$, where for every $(a,b,c)\in A\times B\times C$, include the edge $bc$ in $H$ with color $a$ if $ab$ is an edge in $G''$ with color~$c$. Observe that $H$ is properly edge-colored with $(1-\T)n$ colors and each of them appears at least $n$ times. Thus, by \Cref{gao} or \Cref{th1}, $H$ has a rainbow matching of size $(1-\T)n$. This implies that there is a rainbow matching $M'$ in $G''$ using all the vertices in $A$. Since the vertices and colors used in $M$ are not present in $G''$, the union of $M$ and $M'$ gives us a rainbow matching of size $n$ in $G$, as desired. 

\medskip
\noindent\textbf{Case 2:} Suppose $|A| < (1-\T)n$. By definition of $A$, the vertices outside $A$ have degree less than $2(1+\T)n$ in $G$. We consider a subgraph $G'$ of $G$ by independently deleting all the edges incident to $A$ with probability $p = \frac{\e-\T}{1+\e}$. Then, by Chernoff bound combined with the union bound, w.v.h.p. every vertex in $A$ has degree in $G'$ at most $2(1+\e)n\cdot (1-p) + n^{2/3} = 2(1+\T)n + n^{2/3}$. Thus, w.v.h.p. $G'$ has maximum degree at most $2(1+\T)n + n^{2/3}$. Since the graph $G$ contains at most $(1-\T)n$ edges of each color incident to $A$, by Chernoff bound combined with the union bound, w.v.h.p. for every $c\in \mathscr{C}(G)$, the number of edges of color $c$ in $G'$ is at least 
\begin{equation} \label{eq:degree in G'}
(1-\T)n\cdot (1-p) + \T n - n^{2/3} = \left(1 + \frac{\T^2}{1+\T}\right)\frac{1+\T}{1+\e} n - n^{2/3} \ge \left(1+\frac{\T^2}{2}\right)\frac{1+\T}{1+\e} n.
\end{equation}
At this point, we consider a subgraph $G''$ of $G'$ by deleting each color independently with probability $q = 1 - \frac{1+\T^3}{2(1+\e)}$ (where deleting a color means deleting all edges of that color). By Chernoff bound, w.v.h.p. the number of colors survived in $G''$ is $|\mathscr{C}(G'')| \ge 2(1+\e)n\cdot (1-q) - n^{2/3} \ge n$. Using the bound on maximum degree of $G'$, together with Chernoff bound and the union bound, w.v.h.p. the graph $G''$ has maximum degree at most $(2(1+\T)n + n^{2/3})\cdot (1-q) + n^{2/3}\le (1+2\T^3)\frac{1+\T}{1+\e} n$. (Here, we use the union bound over all vertices in $V(G)$ with positive degree, and the number of such vertices is at most $|\mathscr{C}(G)|\cdot 2n \le 5n^2$.) Using this and the lower bound in \eqref{eq:degree in G'} for the number of edges of each color in $\mathscr{C}(G'')$ in $G''$, we can apply \Cref{gao} or \Cref{th1} to obtain that w.v.h.p. there is a rainbow matching of size $|\mathscr{C}(G'')| \ge n$ in $G''$. This proves that there is a rainbow matching of size $n$ in $G$, which finishes the proof of \Cref{th}.
\end{proof}

\noindent\textbf{Remark on the proof of \Cref{th}.} If one wants to improve our proof to obtain a bound of $(1+\e)n$ instead of $2(1+\e)n$ in \Cref{th}, we remark that the argument in Case 2 works as it is (with a modified definition of $A$ which will contain vertices with degree at least $(1+\T)n$). However, the argument in Case 1 is quite wasteful and does not easily extend. Nonetheless, it is possible that a more in-depth stability-type analysis for Case 1 may help us in reaching closer to the bound of $(1+\e)n$ in \Cref{th}.

\section{Rainbow matchings in sparse graphs with non-proper coloring} \label{sec:4}
We will prove \Cref{th1} in this section. The first three subsections are devoted to proving the following version of \Cref{th1} with a weaker condition on the edge multiplicities and on the number of edges of the same color incident to any vertex. This result will later be used in Subsection~\ref{subsec:reduction of degree} to prove \Cref{th1}.

\begin{theorem} \label{th1 small codegree}
There exists $\xi > 0$ such that for every $C> 0$, there exists $n_0 > 0$ satisfying the following for every $n\ge n_0$. Suppose $G$ is an edge-colored multigraph with maximum degree at most $n$ such that there are at least $(1 + n^{-\xi})n$ edges of each color. Furthermore, suppose that the multiplicity of $G$ is at most $C$ and at most $C$ edges of the same color are incident to any vertex. Then there is a rainbow matching in $G$ that uses every color.
\end{theorem}

\subsection{Randomized algorithm} \label{sec:algorithm}
In this subsection, we give a randomized algorithm that will be used to prove \Cref{th1 small codegree}. Throughout this and the next two subsections, we work in the framework of \Cref{th1 small codegree}. For this, we assume that $\xi$ is a sufficiently small constant, $C\ge 1$ is a fixed constant, and $n$ is sufficiently large relative to $\xi$ and~$C$. Let~$\e = n^{-\xi}$. By our choice of $\xi$ and $n$, we have $\e \le 1/20$. Assume that we are given an edge-colored multigraph $G$ with maximum degree at most $n$ and multiplicity at most $C$ such that there are exactly $(1 + \e)n$ edges of each color and at most $C$ edges of the same color are incident to any vertex. Let $\mathscr{C}$ denote the set of all colors that appears in $G$, i.e., $\mathscr{C} = \mathscr{C}(G)$. 

We now follow a procedure that constructs a rainbow matching using almost all the colors over several iterations. Each iteration is executed using a randomized algorithm, and we can show that we reach the desired state with positive probability. Then we fix that choice of desirable state to analyze the next iteration. This procedure is in terms of some parameters $\h, \d, \t, \g > 0$, which we fix as follows until the end of Subsection~\ref{sec:formal analysis}. 
\begin{equation} \label{epdeleta}
\h = \e^2, \;\;\;\;\;\; \d = \e^{20}, \;\;\;\;\;\; \t = \frac{1}{\d}, \;\;\;\;\;\; \g = \frac{1}{1+\e}.
\end{equation} 
By our choice of $\xi$ and $n$, we have $\d \le \h \le \e \le \g \le 1$. To track the number of edges in each color and the degree of each vertex, we now introduce two functions $s,g:[0,1) \rightarrow [0,1)$ that satisfy $s(x) = (1-\g x)^2$ and $g(x) = (1-x)(1 - \g x)$ for every $x\in [0,1]$. The readers are advised to consult the next subsection to see an intuitive explanation of why these functions are chosen as above. In what follows, we will always assume that $(1-\h)\t$ is an integer since all our arguments can be repeated with $(1-\h)\t$ replaced by $\floor{(1-\h)\t}$.

We will make sure the following happen after the $t$-th iteration for every $t\in \{0\}\cup [(1-\h)\t]$. 
\stepcounter{propcounter}
\begin{enumerate}[leftmargin=*,label = {\bfseries \emph{\Alph{propcounter}}}]
\item \label{inductive property} There is a set $\mathscr{C}_t \subseteq \mathscr{C}$ such that there is a $(\mathscr{C}\setminus \mathscr{C}_t)$-rainbow matching $M_t$ of $G$, and there is a subgraph $G_t$ of $G$ with $V(G_t)\cap V(M_t) = \emptyset$ and $\mathscr{C}(G_t)\subseteq \mathscr{C}_t$ that satisfies the following two properties. 
\end{enumerate}
\begin{enumerate}[leftmargin=*,label = {\bfseries \emph{\Alph{propcounter}\arabic{enumi}}}]
\item \label{p1} For every $c\in \mathscr{C}_t$, the number of edges with color $c$ in $G_t$ is exactly $s_t := s(t\d) (1+\e)n - t\d^{5/4} n$. 
\item \label{p2} For every vertex $v\in V(G_t)$, the degree of $v$ in $G_t$ is at most $d_t := g(t\d) n + t\d^{3/2} n$.
\end{enumerate}

To prove \Cref{th1 small codegree}, our strategy is to show that \ref{inductive property} holds for $t=(1-\h)\t$. Then, using this, we will prove that the multigraph $G_{(1-\h)\t}$ still has so many edges left in each remaining color relative to the maximum degree of $G_{(1-\h)\t}$ that we can extend the matching $M_{(1-\h)\t}$ to a rainbow matching of $G$ using every color in $\mathscr{C}$ via a straightforward application of \Cref{useAlon} on the multigraph $G_{(1-\h)\t}$. 

For convenience, we now introduce two error-tolerance sequences $\a_t = \frac{\g t\d^{5/4}}{s(t\d)}$ and $\b_t = \frac{t\d^{3/2}}{g(t\d)}$ for every $t\in \{0\}\cup [(1-\h)\t]$. Thus, we have the following. 
\begin{equation} \label{eq:s_t and d_t}
s_t = (1-\a_t) s(t\d) (1+\e)n \;\;\;\;\;\; \text{and} \;\;\;\;\;\; d_t = (1+\b_t) g(t\d) n.
\end{equation}

Note that by taking $\mathscr{C}_0 = \mathscr{C}$ and $G_0 = G$, we ensure that \ref{inductive property} holds for $t=0$. Our strategy is to inductively show that \ref{inductive property} holds for every $t\in \{0\}\cup [(1-\h)\t]$. For this, we first fix some $t$ with $0\le t< (1-\h)\t$ and assume that \ref{inductive property} holds for $t$ for a given set $\mathscr{C}_t \subseteq \mathscr{C}$, a $(\mathscr{C}\setminus \mathscr{C}_t)$-rainbow matching $M_t$ of $G$, and a subgraph $G_t$ of $G$. Our goal is now to prove that \ref{inductive property} holds for $t+1$. We will next provide a randomized algorithm that outputs a set $\mathscr{C}_{t+1}\subseteq \mathscr{C}$, a $(\mathscr{C}\setminus \mathscr{C}_{t+1})$-rainbow matching $M_{t+1}$ in $G$, and a subgraph $G_{t+1}$ of $G$ such that with respect to these objects, \ref{inductive property} holds for $t+1$ with a positive probability. Importantly, this will also show that there are some deterministic choices for these objects so that \ref{inductive property} holds for $t+1$. 

\begin{enumerate}[leftmargin=*,label = {\bfseries Step~\arabic{enumi}}]
\item \label{step1} We activate each color in $\mathscr{C}_t$ independently with probability $\T_t=\frac{\d}{1-t\d}$. (This probability is picked to ensure that the number of colors activated in each iteration is roughly equal to $\delta |\mathscr{C}|$.)
\item \label{step2} We select independently one edge in $G_t$ of each activated color uniformly at random. We denote by $T$ the set of all selected edges.
\item \label{step3} We delete all the vertices corresponding to the edges in $T$ from $G_t$. Deleting vertices from a multigraph also deletes all incident edges from it.
\end{enumerate}

\noindent\textbf{Calculation of maximum probability by which a vertex is deleted in Step 3.} The probability that a fixed edge in $G_t$ is selected in \ref{step2} is exactly $\frac{\T_t}{s_t}$ because the corresponding color gets activated with probability $\T_t$ and if a color is activated, then each of the $s_t$ edges of that color gets selected with the same probability. For all $c \in \mathscr{C}_t$ and $v\in V(G_t)$, denote by $d_c(v)$ the number of edges in $G_t$ incident to $v$ that are of color $c$. Thus, the probability that $v$ is deleted in \ref{step3} is given by $p'_v = 1-\prod_{c \in \mathscr{C}_t} \left(1-d_c(v)\cdot \frac{\T_t}{s_t}\right)$. Since $\sum_{c \in \mathscr{C}_t} d_c(v)$ is equal to $d_{G_t}(v)$, it is easy to see that
\begin{equation} \label{inequality_pv} 
p'_v \le \frac{d_{G_t}(v) \T_t}{s_t}.
\end{equation} 
This together with $d_{G_t}(v)\le d_t$ implies that $p'_v \le d_t \cdot \frac{\T_t}{s_t} = \frac{\d \cdot (1+\b_t) g(t\d) n}{(1-t\d) \cdot (1-\a_t) s(t\d) (1+\e)n} = \frac{\g (1+\b_t) g(t\d) \d}{(1-\a_t) (1-t\d) s(t\d)} =: a_t$. 

\begin{enumerate}[leftmargin=*,label = {\bfseries Step~\arabic{enumi}}]
\setcounter{enumi}{3}
\item \label{step4} We delete each vertex $v\in V(G_t)\setminus V(T)$ independently with probability $p_v$ so that $p'_v + (1-p'_v)p_v = a_t$. Thus, $a_t$ is the total probability by which each vertex of $G_t$ gets deleted among \ref{step3} and \ref{step4} combined. We denote by $B$ all vertices in $V(G_t)$ that are deleted during \ref{step3} and \ref{step4}. (We will define the multigraph $G_{t+1}$ such that $V(G_{t+1}) = V(G_t)\setminus B$.)
\item \label{step5} We denote by $T'$ the set of all edges $e \in T$ such that $e$ is not incident to any edge $e' \in T$ with $e'\neq e$. We now define the matching $M_{t+1}$ to be the union of $M_t$ and $T'$. Let $\mathscr{C}'$ denote the set of all colors that appear in some edge in $T'$. We define $\mathscr{C}_{t+1} = \mathscr{C}_t \setminus \mathscr{C}'$. It is clear that $M_{t+1}$ is $(\mathscr{C}\setminus \mathscr{C}_{t+1})$-rainbow. We then delete all edges from $G_t$ that do not have a color from $\mathscr{C}_{t+1}$.
\item \label{step6} For each color $c\in \mathscr{C}_{t+1}$, if the number of edges of color $c$ is more than $s_t$ after \ref{step5}, then we delete arbitrary edges of that color to ensure that there are exactly $s_t$ edges of color $c$. We finally define $G_{t+1}$ to be the multigraph with vertex set $V(G_t)\setminus B$ where the edge set of $G_{t+1}$ is the set of all survived edges of color in $\mathscr{C}_{t+1}$. It is clear that $V(G_{t+1})\cap V(M_{t+1}) = \emptyset$ and $\mathscr{C}(G_{t+1})\subseteq \mathscr{C}_{t+1}$.
\end{enumerate}

To this end, the only thing remaining is to show that \ref{p1} and \ref{p2} hold for $t+1$ with a positive probability. Before diving into the formal analysis of our randomized algorithm, we devote the next subsection to providing a heuristic for what to expect from the algorithm and giving justifications as to why the functions $s$ and $g$ were chosen in the particular ways. However, it is possible to skip the next subsection if the reader is only interested in formal proofs of \Cref{th1 small codegree,th1}.

\subsection{Intuitive analysis}
In this subsection, we provide a rough, non-rigorous analysis of the randomized algorithm described in the previous subsection. In particular, we will assume that $s,g:[0,1)\rightarrow [0,1)$ are two unknown functions and aim to deduce what good choices for them are by heuristically analyzing our algorithm. We will also ignore the error-tolerance sequences defined in the last subsection and assume \eqref{eq:s_t and d_t} with $\a_t = \b_t = 0$. Thus, we assume the following instead of \ref{p1} and \ref{p2} after the $t$-th iteration.

\begin{enumerate}[leftmargin=*,label = {\bfseries \emph{\Alph{propcounter}\arabic{enumi}}*}]
\item\label{colorclass} For every $c\in \mathscr{C}_t$, the number of edges with color $c$ is $s_t = s(t\d) (1+\e)n$.
\item\label{degreevertex} For every vertex $v\in V(G_t)$, the degree $d_{G_t}(v) = g(t\d) d_G(v)$. 
\end{enumerate}

Clearly, we must have $s(0) = 1$ and $g(0) = 1$. To this end, we track the evolution of the number of edges in each color and the degrees of the vertices as the iterations progress. We do that by modeling these parameters before and after the $(t+1)$-st iteration using a system of differential equations.

Since we set $\a_t = \b_t = 0$, the probability that every vertex $v\in V(G_t)$ gets deleted among \ref{step3} and \ref{step4} is $a_t = \frac{\g g(t\d) \d}{(1-t\d)s(t\d)}$. Thus, the probability that a pair of vertices gets deleted is about $a_t^2$, which we ignore in this intuitive analysis as it is a strictly lower-order term than $a_t$. Hence, for every $c\in \mathscr{C}_t$, the expected number of edges deleted in color $c$ among \ref{step3} and \ref{step4} is approximately $2 a_t s_t = \frac{2\g g(t\d) \d}{(1-t\d)} \cdot (1+\e) n$. This suggests the following behavior.

\begin{equation} \label{eq1}
s'(x) = -2 \g \cdot \frac{g(x)}{1-x}.
\end{equation}

Next, let us estimate the difference $d_{G_{t+1}}(v) - d_{G_t}(v)$ for $v\in V(G_{t+1})$ to get a differential equation for the function $g$. We know that each neighbor of $v$ is deleted among \ref{step3} and \ref{step4} with probability~$a_t$. On the other hand, any edge $e$ incident to $v$ is deleted in \ref{step5} by probability around $\T_t$ because the color of $e$ is activated with this probability in \ref{step1}, and there is a relatively low chance of conflicts in \ref{step5}. Thus, neglecting small error terms due to lack of independence, we expect the following to hold approximately. For convenience, we define a function $h:[0,1) \rightarrow [0,1)$ such that $h(x) = \frac{g(x)}{1-x}$.

\begin{align*} 
d_{G_{t+1}}(v) - d_{G_t}(v) &= - \left(a_t + (1 - a_t)\T_t\right) d_{G_t}(v) \\
g((t+1)\d) d_G(v) - g(t\d) d_G(v) &= - \frac{\d}{1-t\d} \cdot g(t\d) d_G(v) - a_t \left(1 - \frac{\d}{1-t\d}\right) g(t\d) d_G(v) \\
(1 - (t+1)\d) h((t+1)\d) - (1 - t\d) h(t\d) &= - \d h(t\d) - \frac{\g h(t\d)\d}{s(t\d)} \cdot (1 - (t+1)\d) h(t\d) \\
\frac{h((t+1)\d) - h(t\d)}{\d} &= -\g \cdot \frac{h(t\d)^2}{s(t\d)}. 
\end{align*}
The above suggests the following.
\begin{equation} \label{eq2}
h'(x) = -\g \cdot \frac{h(x)^2}{s(x)}.
\end{equation} 

By Equations \eqref{eq1} and \eqref{eq2}, we have that $\frac{dh}{ds} = \frac{h}{2s}$, whose solution is $s = \mathrm{c}h^2$ for some constant $\mathrm{c}$. By the initial conditions that $s(0) = g(0) = h(0) = 1$, we get that $s = h^2$. Now, Equation \eqref{eq2} implies that $h'(x) = -\g$. Solving this with the initial condition $h(0) = 1$, we obtain that $h(x) = 1 - \g x$. Hence, $s(x) = (1 - \g x)^2$ and $g(x) = (1-x)(1 - \g x)$. This completes our justification for why the functions $s$ and $g$ were chosen as they were in the last subsection.

We show in the next subsection that the heuristic analysis in this subsection for the degree of each vertex and the number of edges of each color can be made rigorous using concentration inequalities, provided that $t< (1-\h)\t$. Moreover, assuming \ref{colorclass} and \ref{degreevertex} after the $(1-\h)\t$-th iteration, the number of edges in $G_{(1-\h)\t}$ in each color in $\mathscr{C}_{(1-\h)\t}$ is about $(1-\g(1-\h))^2 (1+\e)n$, and the maximum degree of $G_{(1-\h)\t}$ is at most about $\h(1-\g(1-\h)) n$. As long as $(1-\g(1-\h))^2 (1+\e)n \gg 4e \cdot \h(1-\g(1-\h)) n$, we can extend the $(\mathscr{C}\setminus \mathscr{C}_{(1-\h)\t})$-rainbow matching $M_{(1-\h)\t}$ to a rainbow matching using every color in $G$ by including an edge from each color in $\mathscr{C}_{(1-\h)\t}$ by using \Cref{useAlon}. This inequality is indeed true because $\h = \e^2$ and $\xi$ is sufficiently small and $n$ is sufficiently large relative to $\xi$. This argument will be carefully done at the end of the next subsection to conclude the proof of \Cref{th1 small codegree}.

\subsection{Formal analysis} \label{sec:formal analysis}
In this subsection, we formally analyze our randomized algorithm from Subsection~\ref{sec:algorithm} and finish the proof of \Cref{th1 small codegree}. Recall that the algorithm carries out the $(t+1)$-st iteration for some $0\le t < (1-\h)\t$ and our goal is to show that after executing the algorithm, with positive probability, \ref{p1} and \ref{p2} hold for $t+1$. Throughout this subsection, we assume that $\xi$ is a sufficiently small constant, $C\ge 1$ is a fixed constant, and $n$ is sufficiently large relative to $\xi$ and $C$. We start with observing that the error terms $\a_t$ and $\b_t$ are small.

\begin{observation} \label{clm:alpha beta}
For every $t\in \{0\}\cup [(1-\h)\t]$, we have $\a_t\le 1/4$ and $\b_t \le 1/4$.
\end{observation}

\begin{proof}
Using the definitions of $\a_t$ and $\b_t$, the parameter specifications in \eqref{epdeleta}, the fact that $t\le (1-\h)\t$, and that $n$ is sufficiently large relative to $\xi$, we have
\[
\a_t = \frac{\g t\d^{5/4}}{s(t\d)} \le \frac{(1-\h)\d^{1/4}}{(1-\g(1-\h))^2} \le 1/4
\;\;\;\; \text{and} \;\;\;\;
\b_t = \frac{t\d^{3/2}}{g(t\d)} \le \frac{(1-\h)\d^{1/2}}{\h (1-\g(1-\h))} \le 1/4.
\]
\end{proof}

The next lemma shows that \ref{p1} holds for a fixed color w.v.h.p. 

\begin{lemma} \label{lem2}
For every $c\in \mathscr{C}_t$, the number of edges in $G_t$ of color $c$ survived after \ref{step4} is at least $s_{t+1}$ w.v.h.p. Consequently, for every $c\in \mathscr{C}_{t+1}$, the number of edges in $G_{t+1}$ of color $c$ is exactly $s_{t+1}$ w.v.h.p. 
\end{lemma}

\begin{proof}
Fix a color $c\in \mathscr{C}_t$. Remember that by assumption \ref{p1}, the number of edges in $G_t$ of color $c$ is $s_t$. We will first show the following claim. 

\begin{claim} \label{clm: inside lem2}
The number of edges in $G_t$ of color $c$ deleted during \ref{step3} and \ref{step4} combined is at most $2a_t s_t + \sqrt{n} \log n$ w.v.h.p. 
\end{claim}

\begin{proof}
Let $L$ denote the number of edges of color $c$ that are removed during \ref{step3} and \ref{step4}. Since during \ref{step3} and \ref{step4}, every vertex of $G_t$ gets deleted with probability $a_t$, every edge of $G_t$ gets deleted with probability at most $2a_t$. Thus, $\E(L) \le 2a_ts_t$. Note that
\begin{equation} \label{eq:a_t s_t}
a_t s_t = \frac{\g (1+\b_t) g(t\d) \d}{(1-\a_t) (1-t\d) s(t\d)} \cdot (1-\a_t) s(t\d) (1+\e)n = (1+\b_t)(1-\g t\d) \d n.
\end{equation}
Using this along with \Cref{clm:alpha beta}, we obtain $\E(L)\le 2a_t s_t\le 3n$. To show the concentration of $L$, we use Talagrand's inequality. The random variable $L$ can be seen as a function with domain $\Omega$, where $\Omega=\prod_{\mathrm{c} \in \mathscr{C}_t} \Omega_{\mathrm{c}} \times \prod_{v\in V(G_t)} \Omega_v$ with $\Omega_{\mathrm{c}}$ and $\Omega_v$ denoting the probability space for the decision made for the color $\mathrm{c}$ in \ref{step1} and \ref{step2} and the vertex $v$ in \ref{step4}, respectively. If the decision for one of the $\Omega_{\mathrm{c}}$'s or $\Omega_v$'s is altered, then $L$ will be affected by at most $2C$. Thus, $L$ is $2C$-Lipschitz. Furthermore, observe that $L$ is $1$-certifiable. Indeed, for any edge $e$ of color $c$ removed during \ref{step3} and \ref{step4}, there must either be another edge incident to $e$ in $T$ or one of the vertices in $e$ is deleted in Step 4, and these events certify the removal of $e$. Let $\l = \sqrt{n} \log n$. Using \Cref{talagrand} and the fact that $\E(L)\le 3n$, we conclude that the probability that $L$ deviates from its mean by $\l > \frac{\l}{2} + 20 \cdot 2C \cdot \sqrt{\E(L)} + 64\cdot (2C)^2$ is at most $4 \exp\left(-\frac{(\l/2)^2}{8 (2C)^2 (\E(L) + \l/2)}\right) = \exp(-\omega(\log n))$. This finishes the proof of \Cref{clm: inside lem2}.
\end{proof}

Now, using \eqref{eq:a_t s_t} and $t<(1-\h)\t$ and the parameter relations in \eqref{epdeleta} and \ref{p1}, we have the following. 
\begin{align*} 
s_t - 2a_t s_t - \sqrt{n}\log n 
&= s(t\d)(1+\e)n - t\d^{5/4}n - 2(1+\b_t)(1-\g t\d)\d n - \sqrt{n}\log n \\
&= \left((1 - \g t\d)^2 - 2\g\d (1 - \g t\d)\right) (1+\e)n - t\d^{5/4}n - 2\b_t (1-\g t\d) \d n - \sqrt{n}\log n \\
&= \left((1 - \g\d(t+1))^2 - \g^2\d^2\right) (1+\e)n - t\d^{5/4}n - \frac{2t\d^{3/2}\cdot \d n}{1-t\d} - \sqrt{n}\log n \\
&\ge s((t+1)\d)\cdot (1+\e)n - \d^2 n - t\d^{5/4}n - \frac{2\d^{3/2}n}{\h} - \d^2 n \\
&\ge s((t+1)\d)\cdot (1+\e)n - t\d^{5/4}n - \d^{5/4} n 
\;\;\;\;\;\; = s_{t+1}.
\end{align*}
This together with \Cref{clm: inside lem2} implies \Cref{lem2}.
\end{proof}

The next lemma shows that \ref{p2} holds for a fixed vertex w.v.h.p. 

\begin{lemma} \label{lem3}
For every vertex $v\in V(G_t)$, w.v.h.p. it holds that if $v\in V(G_{t+1})$, then $d_{G_{t+1}}(v)\le d_{t+1}$. 
\end{lemma}

\begin{proof}
First of all, keep in mind that \ref{step6} can only decrease the degrees of any vertex; thus, it is enough to prove \Cref{lem3} before executing \ref{step6}. Throughout the proof of this lemma, to simplify our analysis, we use an alternative implementation of our algorithm that produces the same outcomes (more precisely, the outcomes have the same probability distribution under both processes). We will do the first four steps in a different order, namely \ref{step2} $\rightarrow$ \ref{step4} $\rightarrow$ \ref{step1} $\rightarrow$ \ref{step3}. This can be done because these steps can essentially be done independently. We next elaborate on the exact process executed in this way. 
\begin{enumerate}[leftmargin=*,label = {\bfseries Step~\arabic{enumi}*}]
\item \label{step1'} For every color $c\in \mathscr{C}_t$, we select independently one edge of color $c$ in $G_t$ uniformly at random (instead of just for the activated colors as originally described in our algorithm). We denote by $H$ the multigraph with vertex set $V(G_t)$ and edge set containing all selected edges in this step.
\item \label{step2'} We independently select each vertex $v\in V(G_t)$ with probability $p_v$, where $p_v$ is given as in \ref{step4}. We denote by $V$ the set of all selected vertices.
\item \label{step3'} We activate each color in $\mathscr{C}_t$ independently with probability $\T_t = \frac{\d}{1-t\d}$. For each activated color $c$, we pick the edge from $H$ of color $c$. We denote by $T$ the set of all picked edges.
\item \label{step4'} We delete all the vertices in $V\cup V(T)$. (Here, $V(T)$ denotes the set of all vertices used in $T$.) 
\end{enumerate}
Then, \ref{step5} and \ref{step6} are done as described in the original algorithm. The added benefit to performing the steps in the above order is that, typically, the multigraph $H$ in \ref{step1'} has low maximum degree. Therefore, appropriately conditioned on this, when we activate colors in \ref{step3'}, we obtain a good certifiability constant for the application of Talagrand's inequality.

For every $u,v\in V(G_t)$, let $m(u,v)$ denote the number of parallel edges in $G_t$ between $u$ and $v$. In particular, if $u,v$ is non-adjacent in $G_t$, then $m(u,v)=0$. Note that $d_{G_t}(v) = \sum_{u\in N_{G_t}(v)} m(u,v) = \sum_{u\in V(G_t)} m(u,v)$. Let $X_v = \sum_{u\in V} m(u,v)$ and $L_v = \sum_{u\in V(T)\setminus V} m(u,v)$ and $L'_v$ denote the number of edges $uv$ incident to $v$ such that $u\notin V\cup V(T)$ but the color of the edge $uv$ belongs to the set $\mathscr{C}'$ in \ref{step5}. For every $v\in V(G_t)$, we define $D_v = d_{G_t}(v) - (X_v + L_v + L'_v)$. Observe that if $v\in V(G_{t+1})$, then $d_{G_{t+1}}(v) \le D_v$. Thus, to prove \Cref{lem3}, it suffices to show that for every vertex $v\in V(G_t)$, w.v.h.p. $D_v\le d_{t+1}$. 

To this end, we fix a vertex $v\in V(G_t)$ in the remainder of this proof. By our assumption \ref{p2}, we have $d_{G_t}(u) \le d_t$ for every $u\in V(G_t)$. We start by collecting a few w.v.h.p. events in the following claim that will be helpful to prove \Cref{lem3}. Recall that we say a vertex $u\in V(G_t)$ meets a color $c\in \mathscr{C}_t$ if there is an edge in $G_t$ incident to $u$ that is colored with $c$. Define $U$ to be the set of all vertices $u\in V(G_t)$ such that $u$ and $v$ meet some common color. (Notice that $N_{G_t}(v)\subseteq U$). Define $S$ to be the set of all edges $uv\in E(G_t)$ incident to $v$ such that the edge in $H$ with the same color as $uv$ is incident to $u$. Let $\overline{V}$ denote the set $V(G_t)\setminus V$. (For the following claim, recall $p_u$ as defined in \ref{step4}.) 

\begin{claim} \label{nice_events}
The following events hold w.v.h.p.
\stepcounter{propcounter}
\begin{enumerate}[leftmargin=*,label = {\bfseries \emph{\Alph{propcounter}\arabic{enumi}}}]
\item \label{F_1} $d_H(u) \le \log^2 n$ for every vertex $u \in U$.
\item \label{F_2} $|S| \le \log^2 n$.
\item \label{F_3} $\sum_{u\in \overline{V}} m(u,v) \le \sum_{u\in V(G_t)} m(u,v) (1-p_u)  + \d^{7/4} n$.
\item \label{F_4} $\T_t \sum_{u\in \overline{V}} m(u,v) d_H(u) \ge \frac{\T_t}{s_t} \sum_{u \in V(G_t)} m(u,v) (1-p_u) d_{G_t}(u) - \d^{7/4} n$.
\end{enumerate}
\end{claim}

\begin{proof}
By \Cref{clm:alpha beta} and the fact that $g(x)\le s(x)$ for all $x\in [0,1)$, we deduce that $d_t\le 2s_t$. This inequality will be used a few times in this proof.

To see that \ref{F_1} holds w.v.h.p., first observe that each edge of $G_t$ is selected in \ref{step1'} with probability exactly $\frac{1}{s_t}$. Thus, for each vertex $u\in V(G_t)$, we have $\E(d_H(u)) \le \frac{d_t}{s_t} \le 2$. Clearly, $d_H(u)$ is a sum of independent Bernoulli random variables $B_c$ with $c\in \mathscr{C}_t$, where $B_c = 1$ if and only if there is an edge of color $c$ in $H$ incident to $u$. Thus, using the moreover part of Chernoff bound, we have that $\P[d_H(u) \ge \log^2 n] \le \exp(-\log^2 n)$. By definition of $U$, we have $|U| \le d_{G_t}(v)\cdot 2s_t \le 2(1+\e)n^2$. Now, by the union bound over all vertices in $U$, we conclude that \ref{F_1} holds w.v.h.p. 

For \ref{F_2}, note that for every $u\in V(G_t)$ and $c\in \mathscr{C}_t$, the probability that there is an edge of color $c$ in $H$ incident to $u$ is at most $\frac{C}{s_t}$. Thus, $\E(|S|) \le d_t \cdot \frac{C}{s_t} \le 2C$. Clearly, $|S|$ is a random variable with domain $\Omega = \prod_{c \in \mathscr{C}_t} \Omega_c$ where $\Omega_c$ denotes the probability space for the chosen random edge of color $c$ at \ref{step1'}. It is easy to check that $|S|$ is $2C$-Lipschitz and $1$-certifiable. Thus, setting $\l = \frac{1}{2}\log^2 n$ and applying \Cref{talagrand}, we conclude that $|S|$ deviates from its expectation by at least $\l > \l/2 + 20\cdot 2C\cdot \sqrt{\E(|S|)} + 64\cdot (2C)^2$ is at most $4\exp\left(-\frac{(\l/2)^2}{8 (2C)^2 (\E(|S|) + \l/2)}\right) = \exp(-\omega(\log n))$. Thus, w.v.h.p. it holds that $|S| \le 2C + \frac{1}{2}\log^2 n \le \log^2 n$, i.e., \ref{F_2} holds. 

For \ref{F_3}, recall that $X_v = \sum_{u\in V} m(u,v)$ and thus, $X_v + \sum_{u\in \overline{V}} m(u,v) = \sum_{u\in V(G_t)} m(u,v)$. Thus, to prove \ref{F_3} holds w.v.h.p., it suffices to show that w.v.h.p. $X_v \ge \sum_{u\in V(G_t)} p_u m(u,v) - \d^{7/4} n$. To do so, note that using \Cref{clm:alpha beta}, we have 
\[
\E(X_v) = \sum_{u \in V(G_t)} p_u m(u,v) \le d_t = (1+\b_t)g(t\d)n \le 2n.
\] 
Clearly, $X_v$ is a random variable with domain $\Omega = \prod_{u \in V(G_t)} \Omega_u$ where $\Omega_u$ denotes the probability space for whether the vertex $u$ is selected in $V$ at \ref{step2'}. Note that $X_v$ is $C$-Lipschitz and $1$-certifiable. Thus, setting $\l = \sqrt{n}\log n$ and applying \Cref{talagrand}, we conclude that $X_v$ deviates from its expectation by at least $\l > \l/2 + 20C \sqrt{\E(X_v)} + 64C^2$ is at most $4\exp\left(-\frac{(\l/2)^2}{8 C^2 (\E(X_v) + \l/2)}\right) = \exp(-\omega(\log n))$. This, together with the fact that $\sqrt{n}\log n \le \d^{7/4} n$, shows that \ref{F_3} holds w.v.h.p.

For \ref{F_4}, we first consider the random variable $Y= \sum_{u \in V(G_t)} m(u,v) (1-p_u) d_H(u)$. Using \Cref{clm:alpha beta}, we have
\[
\E(Y) = \sum_{u \in V(G_t)} m(u,v) (1-p_u) \frac{d_{G_t}(u)}{s_t} \le \frac{d_t^2}{s_t} \le 2d_t = 2(1+\b_t) g(t\d) n \le 3n.
\]
Clearly, $Y$ is a random variable with domain $\Omega = \prod_{c \in \mathscr{C}_t} \Omega_c$ where $\Omega_c$ denotes the probability space for the chosen random edge of color $c$ at \ref{step1'}. It is easy to check that $Y$ is $2C$-Lipschitz. Note that using \Cref{clm:alpha beta} and $t< (1-\h)\t$ and the parameter relations in \eqref{epdeleta}, for every $u\in V(G_t)$, we have 
\[
p_u \le a_t = \frac{\g (1+\b_t) g(t\d) \d}{(1-\a_t) (1-t\d) s(t\d)} \le \frac{2\d}{1-\g(1-\h)} \le \frac{2\d}{\h} \le \frac{1}{2}.
\]
Thus, $Y$ is $2$-certifiable. To see this, observe that if $Y\ge s$, then there must exist a set of at most $2s$ colors $c\in \mathscr{C}_t$ such that the selected edge in $H$ of color $c$ is incident to some vertex $u\in N_{G_t}(v)$ and these selected edges certify $Y\ge s$. Thus, setting $\l = \sqrt{n}\log n$ and applying \Cref{talagrand}, we conclude that $Y$ deviates from its expectation by at least $\l > \l/2 + 20\cdot 2C\cdot \sqrt{2\E(Y)} + 64\cdot (2C)^2\cdot 2$ is at most $4\exp\left(-\frac{(\l/2)^2}{8\cdot (2C)^2\cdot 2\cdot (\E(Y)+\l/2)}\right) = \exp(-\omega(\log n))$. Thus, the following holds w.v.h.p. 
\begin{equation} \label{F'}
Y \ge \frac{1}{s_t} \sum_{u \in V(G_t)} m(u,v) (1-p_u) d_{G_t}(u) - \sqrt{n} \log n.
\end{equation}

Since both \ref{F_1} and \eqref{F'} hold w.v.h.p., it is enough to prove that \ref{F_4} occurs w.v.h.p. conditioned on the event that both \ref{F_1} and \eqref{F'} occur. For this, fix any choice of $H$ in \ref{step1'} such that \ref{F_1} and \eqref{F'} hold. Let $X = \sum_{u \in \overline{V}} m(u,v) d_H(u)$. Observe the following using \ref{F_1}. 
\[
\E(X\; | H) = \sum_{u \in V(G_t)} m(u,v) (1-p_u) d_H(u) \le d_t \log^2 n \le 2n\log^2 n.
\]
The random variable $X$, conditioned on $H$, has domain $\Omega=\prod_{u \in V(G_t)} \Omega_u$ with $\Omega_u$ denoting the probability space for whether the vertex $u$ is selected in $V$ at \ref{step2'}. If the decision for one of the $\Omega_u$'s is altered, then $X$ will be affected by at most $C\log^2 n$ due to \ref{F_1}. Thus, $X$ is $(C\log^2 n)$-Lipschitz. Furthermore, it is easy to see that $X$ is $1$-certifiable. Hence, setting $\l = \frac{1}{2}\sqrt{n}\log^4 n$ and using \Cref{talagrand}, we conclude that the probability that $X$, conditioned on $H$, deviates from its mean by $\l > \l/2 + 20 \cdot C\log^2 n \cdot \sqrt{\E(X\; |H)} + 64\cdot (C\log^2 n)^2$ is at most $4 \exp\left(-\frac{(\l/2)^2}{8 \cdot C\log^2 n \cdot (\E(X\; |H)+\l/2)}\right) = \exp(-\omega(\log n))$. This, together with \eqref{F'}, gives us the following.
\[
\P\left[X \le \frac{1}{s_t}\sum_{u \in V(G_t)} m(u,v) (1-p_u) d_{G_t}(u) - \sqrt{n}\log^4 n \;\Big| H\right] = \exp(-\omega(\log n)).
\]
Since the above holds for any choice of $H$ satisfying \ref{F_1} and \eqref{F'}, using $\T_t \sqrt{n}\log^4 n \le \frac{\d \sqrt{n} \log^4 n}{\h} \le \d^{7/4} n$, we conclude that \ref{F_4} holds w.v.h.p. This finishes the proof of \Cref{nice_events}.
\end{proof}

We now let $\F$ denote the event that \ref{F_1}--\ref{F_4} hold simultaneously. In the following lemma, we condition on $\F$ to establish a w.v.h.p. upper bound on $D_v = d_{G_t}(v) - (X_v + L_v + L'_v)$. 

\begin{claim} \label{expect}
Conditioned on $\F$, w.v.h.p. we have $D_v \le d_t(1-a_t) \left(1 - \T_t\right) + 4\d^{7/4} n$. 
\end{claim}

\begin{proof}
Fix any choice of $H$ in \ref{step1'} and $V$ in \ref{step2'} such that $\F$ holds. Recall the definition of $S$ and $\overline{V}$ mentioned before \Cref{nice_events}. Define $S'$ to be the set of all edges $uv\in V(G_t)$ incident to $v$ such that $u\in \overline{V}$ and $uv\notin S$. For every edge $uv\in S'$, let $\mathcal{E}_{uv}$ denote the event that one of the edges in $H$ incident to $u$ is picked in \ref{step3'} (i.e., $u\in V(T)$), and $\mathcal{E}'_{uv}$ denote the event that the color of $uv$ belongs to the set $\mathscr{C}'$ in \ref{step5}. Note that 
\begin{equation} \label{eq: L_v and L'_v}
\E(L_v\;| H, V) \ge \sum_{uv\in S'} \P[\mathcal{E}_{uv}] \;\;\;\;\;\; \text{and} \;\;\;\;\;\; \E(L'_v\;| H, V) \ge \sum_{uv\in S'} \P[\overline{\mathcal{E}_{uv}} \cap \mathcal{E}'_{uv}].
\end{equation}
Observe that $\P[\mathcal{E}_{uv}]\ge 1-\left(1-\T_t\right)^{d_H(u)} \ge d_H(u)\T_t - d_H(u)^2\T_t^2$. To obtain a lower bound on $\P[\overline{\mathcal{E}_{uv}} \cap \mathcal{E}'_{uv}]$, first fix an edge $uv\in S'$ and then let $e$ denote the edge in $H$ that has the same color as $uv$. (Note that $e$ is not incident to $u$ since $uv\in S'$.) Then, the event $\overline{\mathcal{E}_{uv}} \cap \mathcal{E}'_{uv}$ is equivalent to the following event that
\begin{itemize}
\item the color of $uv$ is activated in \ref{step3'} and 
\item none of the colors used in the edges in $H$ incident to $u$ or $e$ is activated in \ref{step3'}. 
\end{itemize}
Observe that the first of the above events and the event that none of the edges incident to $e$ is picked in \ref{step3'} ensure that $e$ is added to $T'$ and also to the matching $M_{t+1}$ in \ref{step5}; hence, the color of $uv$ belongs to $\mathscr{C}'$ in \ref{step5}. Thus, we have $\P[\overline{\mathcal{E}_{uv}} \cap \mathcal{E}'_{uv}] \ge \T_t\left(1- \T_t\right)^{3\log^2 n}$, because the number of edges in $H$ that are incident to $u$ or $e$ is at most $3\log^2 n$ (due to \ref{F_1}). Hence, using \eqref{eq: L_v and L'_v}, we have

\begin{align} 
\E(L_v + L'_v \;| H, V) &\ge \sum_{uv\in S'} \left(\P[\mathcal{E}_{uv}] + \P[\overline{\mathcal{E}_{uv}} \cap \mathcal{E}'_{uv}]\right) \nonumber \\
&\ge \sum_{u \in \overline{V}} m(u,v) \cdot \left(d_H(u)\T_t - d_H(u)^2\T_t^2 + \T_t \left(1-\T_t\right)^{3\log^2 n}\right) -\sum_{uv\in S} 1 \nonumber \\
&\ge \sum_{u \in \overline{V}} m(u,v) \cdot \left(d_H(u)\T_t - \log^4 n \cdot \T_t^2 + \T_t - 3\log^2 n\cdot \T_t^2\right) -|S| \label{size_lu} \\
&\ge \sum_{u \in \overline{V}} m(u,v) \cdot \left(\T_t + d_H(u)\T_t\right) - d_t \T_t^2 (\log^4 n + 3\log^2 n) - \log^2 n \label{size_z} \\
&\ge \sum_{u \in \overline{V}} m(u,v) \cdot \left(\T_t + d_H(u)\T_t\right) - \d^{7/4} n. \label{dtheta}
\end{align}
In the above, \eqref{size_lu} uses the fact that $d_H(u)\le \log^2 n$ (due to \ref{F_1}). In \eqref{size_z}, we use $\sum_{u\in \overline{V}} m(u,v)\le d_{G_t}(v)\le d_t$ and $|S|\le \log^2 n$ (due to \ref{F_2}). In \eqref{dtheta}, we use the fact that $n$ is sufficiently large and $d_t\T_t^2 = (1+\b_t)g(t\d) n \cdot \frac{\d^2}{(1-t\d)^2} \le \frac{2\d^2 n}{\h} = 2\d^{19/10} n$, where the inequality uses \Cref{clm:alpha beta} and $t<(1-\h)\t$. 

We now let $L = L_v + L'_v$ and, conditioning on $H, V$, we prove a concentration bound on $L$ using Talagrand's inequality. The random variable $L$ conditioned on $H, V$ can be seen as a function with domain $\Omega=\prod_{c \in \mathscr{C}} \Omega_c$ with $\Omega_c$ denoting the probability space for the decision made for color $c$. If the decision for one of the $\Omega_c$'s is altered, then $L$ will be affected by at most $5C$. Thus, $L$ is $5C$-Lipschitz. Furthermore, we claim that $L$ is $(3\log^2 n)$-certifiable. Indeed, the event $u\in V(T)$ is certified in \ref{step3'} by the activation of a color of some edge in $H$ incident to $u$. Moreover, the event that the color of the edge $uv$ belongs to $\mathscr{C}'$ in \ref{step5} is certified by the activation of the color of $uv$ in \ref{step3'} and the non-activation of the colors (whose number is less than $2\log^2 n$ due to \ref{F_1}) corresponding to the edges in $H$ incident to $e$ except itself, where $e$ is the edge in $H$ with the same color as $uv$. This establishes that $L$ is $(3\log^2 n)$-certifiable. Note that by definition of $L$, we have $\E(L \;| H, V) \le d_{G_t}(v) \le d_t \le 2n$. Hence, setting $\l = \sqrt{n} \log^4 n$ and using \Cref{talagrand}, we conclude that the probability that $L$, conditioned on $H$ and $V$, deviates from its mean by $\l > \frac{\l}{2} + 20 \cdot 5C \cdot \sqrt{3\log^2 n \cdot \E(L\;| H, V)} + 64\cdot (5C)^2\cdot 3\log^2 n$ is at most $4 \exp\left(-\frac{(\l/2)^2}{8\cdot (5C)^2 \cdot 3\log^2 n \cdot (\E(L \;| H, V) +\l/2)}\right) = \exp(-\omega(\log n))$. Thus, using the fact that $\l \le \d^{7/4} n$, we have 
\begin{equation} \label{L_good_event}
\P\left[L_v + L'_v \ge \E(L_v + L'_v\;| H, V) - \d^{7/4} n \;\big| H, V\right] = 1-\exp(-\omega(\log n)).
\end{equation}
Thus, conditioning on $H, V$ and assuming that the event in \eqref{L_good_event} holds, using \eqref{dtheta}, we have
\begin{align}
D_v &= d_{G_t}(v) - (X_v + L_v + L'_v) = \sum_{u\in \overline{V}} m(u,v) - (L_v + L'_v) \nonumber \\
&\le \sum_{u \in \overline{V}} m(u,v)\cdot \left(1 - \T_t - d_H(u)\T_t\right) + 2\d^{7/4} n \nonumber \\
&= \left(1 - \T_t\right) \sum_{u \in \overline{V}} m(u,v) - \T_t \sum_{u \in \overline{V}} m(u,v) d_H(u) + 2\d^{7/4} n \nonumber \\
&\le \left(1 - \T_t\right) \sum_{u\in V(G_t)} m(u,v) (1 - p_u) - \T_t \sum_{u\in V(G_t)} m(u,v)(1-p_u)\cdot \frac{d_{G_t}(u)}{s_t} + 4\d^{7/4} n \label{conc} \\
&= \sum_{u\in V(G_t)} m(u,v)(1-p_u) \left(1 - \T_t - \frac{d_{G_t}(u) \T_t}{s_t}\right) + 4\d^{7/4} n \nonumber \\
&\le \sum_{u\in V(G_t)} m(u,v)(1-p_u) \left(1 - \frac{d_{G_t}(u) \T_t}{s_t}\right) \left(1 - \T_t\right) + 4\d^{7/4} n \nonumber \\
&\le \sum_{u\in V(G_t)} m(u,v)(1-p_u) \left(1 - p'_u\right) \left(1 - \T_t\right) + 4\d^{7/4} n \label{ineq} \\
&= \sum_{u\in V(G_t)} m(u,v)(1-a_t) \left(1 - \T_t\right) + 4\d^{7/4} n 
\;\;\;\;\;\; \le (1-a_t) \left(1 - \T_t\right) d_t + 4\d^{7/4} n. \label{degree of v}
\end{align}
In the above, for \eqref{conc}, we use the facts that the choices of $H$ and $V$ satisfy \ref{F_3} and \ref{F_4}. In \eqref{ineq}, we use the inequality at \eqref{inequality_pv}. In \eqref{degree of v}, we use the definition of $p_u$ mentioned in \ref{step4} and also the fact that $\sum_{u\in V(G_t)} m(u,v) = d_{G_t}(v) \le d_t$.

Since the bounds in \eqref{L_good_event} and \eqref{degree of v} hold for any choices of $H$ and $V$ satisfying $\F$, we have the following. 
\begin{align*}
\P\left[D_v \le (1-a_t) \left(1 - \T_t\right) d_t + 4\d^{7/4} n \;\big| \F\right] = 1-\exp(-\omega(\log n)).
\end{align*}
This finishes the proof of \Cref{expect}. 
\end{proof}

Now, note that $g(t\d) \ge g((t+1)\d)$ and thus $\b_t\le \frac{t\d^{3/2}}{g((t+1)\d)}$. Thus, observe that
\begin{align*}
\left(1 - a_t\right)\left(1 - \T_t\right) d_t &= \left(1 - \frac{(1+\b_t)\cdot \g\d(1-\g t\d)}{(1-\a_t)\cdot (1-\g t\d)^2}\right) \left(1 - \frac{\d}{1-t\d}\right) (1+\b_t) (1-t\d) (1 - \g t\d) n \\
&\le \left(1 - \frac{\g\d(1-\g t\d)}{(1-\g t\d)^2}\right) \left(1 - \frac{\d}{1-t\d}\right) (1+\b_t) (1-t\d) (1 - \g t\d) n \\
&= (1+\b_t) g((t+1)\d) n \;\;\;\;\;\; 
\le g((t+1)\d) n + t\d^{3/2} n.
\end{align*}
Using the above inequality and the equality that $d_{t+1} = g((t+1)\d) n + (t+1)\d^{3/2} n$, we have
\[
d_{t+1} - (\left(1 - a_t\right)\left(1 - \T_t\right) d_t + 4\d^{7/4} n) \ge \d^{3/2} n - 4\d^{7/4} n > 0.
\]
This, together with \Cref{nice_events,expect}, finishes the proof of \Cref{lem3}. 
\end{proof}

Our goal is to now combine \Cref{lem2,lem3} to show that, with positive probability, \ref{p1} and \ref{p2} hold for $t+1$. Showing this using the union bound is impossible, since we do not assume any upper bound on the number of colors or vertices in the multigraph $G$. However, utilizing the local lemma, we show below that with positive probability, \Cref{lem2} holds for every color and \Cref{lem3} holds for every vertex simultaneously.

\begin{lemma} \label{locallemma}
The following happens simultaneously with positive probability. 
\begin{itemize}
\item \label{1stcondition} For every $c\in \mathscr{C}_{t+1}$, the number of edges in $G_{t+1}$ of color $c$ is $s_{t+1}$. 
\item \label{2ndcondition} For every $v\in V(G_{t+1})$, the degree of $v$ in $G_{t+1}$ is at most $d_{t+1}$. 
\end{itemize}
\end{lemma}

\begin{proof}
Let $\G$ be the graph with vertex set $\mathscr{C}_t$, where two colors are joined by an edge if a vertex in $G_t$ meets both colors. For each $c \in \mathscr{C}_t$, we define the `bad' event $\mathcal{B}_c$ to be the event that $c\in \mathscr{C}_{t+1}$ and at least one of the following happens.
\begin{itemize}
\item For the color $c$, the first assertion of \Cref{locallemma} does not hold. 
\item There exists $v\in V(G_{t+1})$ that meets the color $c$ such that $v$ violates the second assertion of \Cref{locallemma}.
\end{itemize}
It is straightforward to see that to prove \Cref{locallemma}, it suffices to show that with positive probability, none of the events $\mathcal{B}_c$ happens. Note that for each $c \in \mathscr{C}_t$, the event $\mathcal{B}_c$ is determined by the random choices involving the color $c'$ with $\dist_{\G}(c, c') \le 2$, where $\dist_{\G}(c, c')$ denotes the distance between $c$ and $c'$ in the graph $\G$, i.e., the number of edges in a shortest path between $c$ and $c'$. We claim that the graph $\G^4$, defined as the graph on $\mathscr{C}_t$ with edges between $c$ and $c'$ if and only if $\dist_{\G}(c, c') \le 4$, is a dependency graph for the events $\{\mathcal{B}_c\}_{c\in\mathscr{C}_t}$. This is because if $cc' \notin E(\G^4)$, then the events $\mathcal{B}_c$ and $\mathcal{B}_{c'}$ are determined by disjoint random choices. Observe that the maximum degree of $\G$ is at most $4n^2$, and consequently, the maximum degree of $\G^4$ is at most $256n^8$. By \Cref{lem2,lem3} and the union bound, for all $c \in \mathscr{C}_t$, we have $\P[\mathcal{B}_c] = \exp(-\omega(\log n))$. Thus, by the local lemma (i.e., \Cref{lll}), we conclude that with positive probability, none of the events $\mathcal{B}_c$ happens, as desired. This finishes the proof of \Cref{locallemma}.
\end{proof}

\Cref{locallemma} shows that \ref{p1} and \ref{p2} hold for $t+1$ with positive probability with respect to the objects generated by our randomized algorithm. Thus, there are deterministic choices for the objects in \ref{inductive property} so that it holds for $t+1$. This finishes our inductive argument to prove that \ref{inductive property} holds for every $t\in \{0\}\cup [(1-\h)\t]$.

\begin{proof}[{\bf Wrapping up the proof of \Cref{th1 small codegree}}]
We have shown above that \ref{inductive property} holds for $t=(1-\h)\t$. This means that there is a set $\mathscr{C}_{(1-\h)\t} \subseteq \mathscr{C}$ such that there is a $(\mathscr{C}\setminus \mathscr{C}_{(1-\h)\t})$-rainbow matching $M_{(1-\h)\t}$ of $G$, and there is a subgraph $G_{(1-\h)\t}$ of $G$ with $V(G_{(1-\h)\t})\cap V(M_{(1-\h)\t}) = \emptyset$ and $\mathscr{C}(G_{(1-\h)\t})\subseteq \mathscr{C}_{(1-\h)\t}$ such that \ref{p1} and \ref{p2} hold for $t=(1-\h)\t$. Using \Cref{clm:alpha beta}, the parameter specifications in \eqref{epdeleta}, and the fact that $n$ is sufficiently large relative to $\xi$, we have 
\[
\frac{s_{(1-\h)\t}}{d_{(1-\h)\t}} = \frac{(1-\a_{(1-\h)\t}) (1-\g(1-\h))^2 (1+\e)n}{(1+\b_{(1-\h)\t}) (1-\g(1-\h)) \h n} \ge \frac{(3/4)\cdot \e}{(5/4)\cdot \h} = \frac{3}{5\e} \ge 4e.
\]
Thus, we invoke \Cref{useAlon} to find a $\mathscr{C}_{(1-\h)\t}$-rainbow matching $M$ in the multigraph $G_{(1-\h)\t}$. Thus, the multigraph $G$ contains the $\mathscr{C}$-rainbow matching obtained by taking the union of $M$ and $M_{(1-\h)\t}$. This finishes the proof of \Cref{th1 small codegree}.
\end{proof}

\subsection{Reducing local degree} \label{subsec:reduction of degree}
In this subsection, we deduce \Cref{th1} from \Cref{th1 small codegree}. We utilize the following two lemmas that are inspired by similar results from \cite{A92,CT,LS}.

\begin{lemma} \label{lem:random partition}
There exist $n_0, \xi_0 > 0$ such that the following hold for every $n\ge n_0$ and $0< \xi \le \xi_0$. Suppose $G$ is an edge-colored multigraph with maximum degree at most $n$ such that there are at least $(1 + n^{-\xi})n$ edges of each color. Furthermore, suppose that the multiplicity of $G$ is at most $n^{1/3}$ and at most $n^{1/3}$ edges of the same color are incident to any vertex. Then there is a subgraph $G'$ of $G$ such that the following hold.
\stepcounter{propcounter}
\begin{enumerate}[leftmargin=*,label = {\bfseries \emph{\Alph{propcounter}\arabic{enumi}}}]
\item \label{eq:random part1} For every $c\in \mathscr{C}(G)$, there are at least $\left(1+3n^{-\xi}/4\right)n^{1/3}$ edges of color $c$ in $G'$.
\item \label{eq:random part2} The maximum degree of $G'$ is at most $\left(1+n^{-\xi}/4\right)n^{1/3}$.
\item \label{eq:random part3} The multiplicity of $G'$ is at most $5$.
\item \label{eq:random part4} At most $5$ edges of the same color are incident to any vertex in $G'$.
\end{enumerate}
\end{lemma}

\begin{proof}
Let $\xi$ be a sufficiently small positive constant and $n$ be a sufficiently large integer. Without loss of generality, we assume that the graph $G$ contains exactly $(1 + n^{-\xi})n$ edges of each color. Consider the subgraph $G'$ of $G$ obtained by retaining all edges of $G$ independently with probability $p = 1/n^{2/3}$. For every $c\in \mathscr{C}(G)$, let $G'_c$ denote the subgraph of $G'$ induced by the set of all edges of color $c$ in $G'$. We aim to apply the local lemma to show that, with positive probability, the properties in \ref{eq:random part1}--\ref{eq:random part4} hold simultaneously.

We define four bad events corresponding to each of \ref{eq:random part1}--\ref{eq:random part4}. For every $c\in \mathscr{C}(G)$, let $\mathcal{A}_c$ denote the event that $e(G'_c) < \left(1+3n^{-\xi}/4\right)n^{1/3}$. For every $v\in V(G)$, let $\mathcal{B}_v$ denote the event that $d_{G'}(v) > \left(1+n^{-\xi}/4\right)n^{1/3}$. For every pair of adjacent vertices $u,v\in V(G)$, let $\mathcal{C}_{u,v}$ denote the event that the multiplicity of $(u,v)$ in $G'$ is at least $6$. For every $v\in V(G)$ and $c\in \mathscr{C}(G)$ such that $v$ meets $c$, let $\mathcal{D}_{v,c}$ denote the event that $d_{G'_c}(v) \ge 6$. 

The event $\mathcal{A}_c$ is completely determined by the choices for edges of color $c$ in $G$, the event $\mathcal{B}_v$ is determined by the choices for the edges incident to $v$, the event $\mathcal{C}_{u,v}$ is determined by the choices for (multi-)edges $uv$, and the event $\mathcal{D}_{v,c}$ is determined by the choices for the edges of color $c$ incident to $v$. Thus, each of these events is mutually independent of all but at most $O(n)$ other events.

We now show that the probabilities of the above bad events are small. For every $c\in \mathscr{C}(G)$, since $e(G'_c)$ is distributed as a binomial random variable with mean $(1+n^{-\xi})np = (1+n^{-\xi})n^{1/3}$, it follows from Chernoff bound that the event $\mathcal{A}_c$ holds with probability at most $\exp\left(-\Omega(n^{-2\xi}n^{1/3})\right) \le n^{-2}$. Similarly, for every $v\in V(G)$, since $d_{G'}(v)$ is distributed as a binomial random variable with mean at most $np = n^{1/3}$, it follows from Chernoff bound that the event $\mathcal{B}_v$ holds with probability at most $\exp\left(-\Omega(n^{-2\xi}n^{1/3})\right) \le n^{-2}$. Finally, since the multiplicity of $G$ is at most $n^{1/3}$ and at most $n^{1/3}$ edges of the same color are incident to any vertex in $G$, both the probability that $\mathcal{C}_{u,v}$ holds and the probability that $\mathcal{D}_{v,c}$ holds are at most $\binom{n^{1/3}}{6} p^6 \le n^{-2}$. Thus, by the local lemma, none of the bad events occur with positive probability. Thus, there exists a subgraph $G'$ of $G$ satisfying \ref{eq:random part1}--\ref{eq:random part4}. 
\end{proof}

\begin{lemma} \label{lem:equal partition}
There exists $D_0$ such that the following holds for every $D\ge D_0$ and $d > \log^4 D$. Suppose $G$ is an edge-colored multigraph with maximum degree at most $D$ such that there are $2n$ edges of each color. Furthermore, suppose that the multiplicity of $G$ is at most $d$ and at most $d$ edges of the same color are incident to any vertex. Then there is a subgraph $G'$ of $G$ such that the following hold.
\stepcounter{propcounter}
\begin{enumerate}[leftmargin=*,label = {\bfseries \emph{\Alph{propcounter}\arabic{enumi}}}]
\item \label{eq:equal part1} For every $c\in \mathscr{C}(G)$, there are $n$ edges of color $c$ in $G'$.
\item \label{eq:equal part2} The maximum degree of $G'$ is at most $D' = D/2 + D^{2/3}$.
\item \label{eq:equal part3} The multiplicity of $G'$ is at most $d' = d/2 + d^{2/3}$.
\item \label{eq:equal part4} At most $d' = d/2 + d^{2/3}$ edges of the same color are incident to any vertex in $G'$.
\end{enumerate}
\end{lemma}

\begin{proof}
Let $D$ be a sufficiently large integer and $d > \log^4 D$. For every color $c\in \mathscr{C}(G)$, arbitrarily pair up the $2n$ edges of color $c$. For each such pair of edges, designate one of them randomly and independently of the others to the subgraph $G'$ of $G$. For every $c\in \mathscr{C}(G)$, let $G'_c$ denote the subgraph of $G'$ induced by the set of all edges of color $c$ in $G'$. By construction, the property \ref{eq:equal part1} is satisfied. We will now use the local lemma, like in the proof of the last lemma, to show that with positive probability the properties in \ref{eq:equal part2}--\ref{eq:equal part4} hold.

We define three bad events. For every $v\in V(G)$, let $\mathcal{B}_v$ denote the event that $d_{G'}(v) > D/2 + D^{2/3}$. For every pair of adjacent vertices $u,v\in V(G)$, let $\mathcal{C}_{u,v}$ denote the event that the multiplicity of $(u,v)$ in $G'$ is at least $d/2 + d^{2/3}$. For every $v\in V(G)$ and $c\in \mathscr{C}(G)$ such that $v$ meets $c$, let $\mathcal{D}_{v,c}$ denote the event that $d_{G'_c}(v) \ge d/2 + d^{2/3}$. By a same argument as in the proof of \Cref{lem:random partition}, one can show that each of the bad events $\mathcal{B}_v$, $\mathcal{C}_{u,v}$, and $\mathcal{D}_{v,c}$ is mutually independent of all but at most $O(D)$ other events. 

We now show that the probabilities of the above bad events are small. We first consider the event $\mathcal{B}_v$. Observe that if two edges incident to $v$ are paired up, then exactly one of them will lie in the graph $G'$. Let $E$ be the set of all edges incident to $v$ that are paired to edges not incident to $v$. Let $m$ denote the number of edges in $E$ that are designated to $G'$. Then, $d_{G'}(v) \le m + (D-|E|)/2$. The quantity $m$ is binomially distributed with parameters $|T|\le D$ and $1/2$. Thus, by Chernoff bound, the probability that $m > |E|/2 + D^{2/3}$ is at most $\exp\left(-\Omega\left((D^{2/3})^2/|T|\right)\right) \le D^{-2}$. Consequently, $\P[\mathcal{B}_v] \le D^{-2}$. Similarly, one can use the assumption that $d > \log^4 D$ to prove that $\P[\mathcal{C}_{u,v}] \le D^{-2}$ and $\P[\mathcal{D}_{v,c}] \le D^{-2}$. Thus, by the local lemma, none of the bad events occur with positive probability. This shows that there exists a subgraph $G'$ of $G$ satisfying \ref{eq:equal part1}--\ref{eq:equal part4}. 
\end{proof}

We now finally prove \Cref{th1} using \Cref{th1 small codegree} and \Cref{lem:random partition,lem:equal partition}.

\begin{proof}[{\bf Proof of \Cref{th1}}]
Let $\xi >0$ be sufficiently small and $\e_0 >0$ be sufficiently small relative to $\xi$. Let $0 < \e \le \e_0$ and $\d = \e^{\xi}$. Suppose $G$ is an edge-colored multigraph with maximum degree at most $n$ such that there are exactly $(1 + \d)n$ edges of each color. (If there are more edges of some color, then we arbitrarily delete such edges to ensure this.) Furthermore, suppose that the multiplicity of $G$ is at most $\e n$ and at most $\e n$ edges of the same color are incident to any vertex. Since the assertion of \Cref{th1} is vacuously true when $\e n < 1$, we assume that $n\ge 1/\e$ in the remainder of the proof. We now split into two cases.

\medskip
\noindent\textbf{Case 1:} Suppose $n\le \e^{-4/3}$. Then, since $\e n \le n^{1/4}$ and $\d \ge n^{-\xi}$, we use \Cref{lem:random partition} to find a subgraph $G'$ of $G$ satisfying \ref{eq:random part1}--\ref{eq:random part4}. Let $n' = \left(1+n^{-\xi}/4\right) n^{1/3}$. Using these hypotheses, we then apply \Cref{th1 small codegree} with $C=5$, $n = n'$, and $\xi = 4\xi$ to find a rainbow matching in $G'$ using every color. Indeed, this application is possible because 
\[
\left(1+3n^{-\xi}/4\right)n^{1/3} \ge \left(1+n^{-\xi}/4\right)^2 n^{1/3} = (1+(1/4)\cdot (n^{1/3})^{-3\xi}) n' \ge (1+ (n')^{-4\xi}) n',
\]
where the last step uses the fact that $n'$ is sufficiently large relative to $\xi$ (which follows from the facts that $n\ge 1/\e$ and $\e$ is sufficiently small relative to $\xi$). The rainbow matching we found in $G'$ is clearly also a rainbow matching in $G$ that uses every color, as desired.

\medskip
\noindent\textbf{Case 2:} Suppose $n > \e^{-4/3}$. Then, let $j$ be the positive integer such that $2^{j-1} < \e^{4/3} n \le 2^j$. For every $c\in \mathscr{C}(G)$, we then delete less than $2^j$ edges of color $c$ from $G$ to ensure that $2^j$ divides the number of edges of color $c$. Let $E_c$ denote the set of all remaining edges of color $c$. Then, $2^j$ divides $|E_c|$ and we have $|E_c| \ge (1 + \d)n - 2^j$ for every $c\in \mathscr{C}(G)$. Define $D_0 = n$ and $d_0 = \e n$, and for every $t\ge 0$, define 
\[
D_{t+1} = D_t/2 + D_t^{2/3} \;\;\;\;\;\; \text{and} \;\;\;\;\;\; d_{t+1} = d_t/2 + d_t^{2/3}
\]
We claim the following.
\stepcounter{propcounter}
\begin{enumerate}[leftmargin=*,label = {\bfseries \emph{\Alph{propcounter}\arabic{enumi}}}]
\item \label{eq:max degree} $\frac{1}{2\e^{4/3}} < D_j \le (1+\frac{\d}{2})(\frac{n}{2^j} -1)$,
\item \label{eq:color degree} $d_j \le (\frac{n}{2^j} -1)^{1/3}$, and
\item \label{eq:global vs color} $d_t > \log^4 D_t$ for $0\le t < j$.
\end{enumerate}
We first use these claims to complete the proof of \Cref{th1}. Let $n^* = |E_c|/2^j$ for $c\in \mathscr{C}(G)$. Note that $n^* \ge (1+\d)n/2^j - 1 \ge (1+\d)(n/2^j - 1)$. Using \ref{eq:global vs color} and the lower bound in \ref{eq:max degree}, we invoke \Cref{lem:equal partition} repeatedly $j$ times to get a graph $G'$ such that \ref{eq:equal part1}--\ref{eq:equal part4} hold with $n$, $D'$, and $d'$ replaced by $n^*$, $D_j$, and $d_j$. Let $n' = (1+\frac{\d}{2})(\frac{n}{2^j} -1)$. Since $n/2^j > 1/(2\e^{4/3})$ and $\e$ is sufficiently small relative to $\xi$, we have that $n'$ is sufficiently large relative to $\xi$, and $n'\ge n/2^{j+1}$. Using these and $\d < 1$ and also since \ref{eq:max degree}, \ref{eq:color degree} hold, we have $D_j \le n'$ and $d_j\le n'^{1/3}$ and $n^*\ge (1+\d)(n/2^j - 1) \ge (1+\d/3)n'$ and $\d/3 \ge (1/3)\cdot (2^{j-1}/n)^{3\xi/4} \ge (1/3)\cdot (4n')^{-3\xi/4} \ge (n')^{-\xi}$. We then apply \Cref{lem:random partition} on $G'$ with $n$ replaced by $n'$ to get a graph $G''$ satisfying \ref{eq:random part1}--\ref{eq:random part4} with $n$ replaced by $n'$. Similar to Case 1, we finally apply \Cref{th1 small codegree} to find a rainbow matching in $G''$ using every color. This yields a rainbow matching in $G$ using every color, as desired.

To finish the proof of \Cref{th1}, the only thing remaining is to prove \ref{eq:max degree}--\ref{eq:global vs color}. For the lower bound in \ref{eq:max degree}, observe that whenever $0\le t\le j$, we have
\begin{equation} \label{eq:lower bound for D}
D_t \ge \frac{D_0}{2^t} = \frac{n}{2^t} \ge \frac{n}{2^j} > \frac{1}{2\e^{4/3}}.
\end{equation}
For the upper bound in \ref{eq:max degree}, we first notice that for every $t\ge 0$,
\[
D_{t+1} = \frac{D_t}{2} + D_t^{2/3} \le \frac{(D_t^{1/3} +1)^3}{2}
\;\;\;\;\;\; \text{and so} \;\;\;\;\;\;
D_{t+1}^{1/3}\le \frac{D_t^{1/3}}{2^{1/3}} + \frac{1}{2^{1/3}}.
\]
Thus, 
\[
D_j^{1/3}\le \frac{D_0^{1/3}}{2^{j/3}} + \sum_{t\in [j]} \frac{1}{2^{t/3}}
\le \frac{n^{1/3}}{2^{j/3}} + 4
\le (1+\d/4)^{1/3} \frac{n^{1/3}}{2^{j/3}},
\]
where the last inequality uses $\frac{n}{2^j} > \frac{1}{2\e^{4/3}}$ and $\d = \e^{\xi}$ where both $\e$ and $\xi$ are sufficiently small. Using the same reasoning once more, the upper bound in \ref{eq:max degree} follows as below.
\[
D_j \le \left(1+\d/4\right) \frac{n}{2^j} \le \left(1+\d/2\right)\left(\frac{n}{2^j} -1\right)
\]
Using this, we have the following inequality for every $0\le t< j$, which will be used later.
\begin{equation} \label{eq: backward inequality for D}
D_t \le 2^{j-t} D_j \le 2^{j-t} (1+\d/2) \frac{n}{2^j} \le \frac{2n}{2^t}.
\end{equation}
For the proof of \ref{eq:color degree}, treating $d_t$ the same way as $D_t$, we obtain the following.
\[
d_j^{1/3}
\le \frac{(\e n)^{1/3}}{2^{j/3}} + 4
\le 2 \frac{(\e n)^{1/3}}{2^{j/3}},
\]
where the last inequality uses $\frac{\e n}{2^j} > \frac{1}{2\e^{1/3}}$ and that $\e$ is sufficiently small. Since $\e \le (2^j/n)^{3/4}$ and $n/2^j$ is sufficiently large, we have
\[
d_j \le 8 \frac{\e n}{2^j} \le 8 \left(\frac{n}{2^j}\right)^{1/4} \le \left(\frac{n}{2^j} -1\right)^{1/3}.
\]
Finally, for \ref{eq:global vs color}, observe that $\e \ge (2^{j-1}/n)^{3/4} \ge (2^t/n)^{3/4}$ whenever $0\le t < j$. Thus, using \eqref{eq: backward inequality for D}, we have
\[
d_t \ge \frac{\e n}{2^t} \ge \left(\frac{n}{2^t}\right)^{1/4} \ge \left(\frac{D_t}{2}\right)^{1/4} \ge \log^4 D_t,
\]
where the last inequality uses \eqref{eq:lower bound for D} and that $\e$ is sufficiently small. This finishes the verification of \ref{eq:max degree}--\ref{eq:global vs color} and thus also the proof of \Cref{th1}.
\end{proof}

\section*{Acknowledgements}
The first author thanks Tuan Tran for discussions that led to an improvement of \Cref{th1} from a previous version of this paper. We also thank an anonymous referee for helpful comments that improved the exposition.


\end{document}